\documentclass[a4paper]{amsart}
\usepackage{amsmath,amssymb}
\usepackage{stmaryrd}
\usepackage{tikz}
\usepackage[normalem]{ulem}

\usepackage{color}

\title[On the Uniform Computational Content of Computability Theory]
{On the Uniform Computational Content\\ of Computability Theory}

\author{Vasco Brattka}
\address{Faculty of Computer Science, Universit\"at der Bundeswehr M\"unchen, Germany and 
             Department of Mathematics \& Applied Mathematics, University of Cape Town, South Africa\footnote{Vasco Brattka is supported by the National Research Foundation of South Africa.}} 
\email{Vasco.Brattka@cca-net.de}

\author{Matthew Hendtlass}
\address{University of Canterbury, Christchurch, New Zealand}
\email{Matthew.Hendtlass@canterbury.ac.nz}

\author{Alexander P.\ Kreuzer}
\address{Department of Mathematics, National University of Singapore, Singapore\footnote{Alexander P.\ Kreuzer is supported by the Ministry of Education of Singapore through grant R146-000-184-112 (MOE2013-T2-1-062).}}
\email{matkaps@posteo.de}

\def\AA{{\mathcal A}}

\def\DD{{\mathcal D}}

\def\PP{{\mathcal P}}

\def\IN{{\mathbb{N}}}
\def\IQ{{\mathbb{Q}}}
\def\IR{{\mathbb{R}}}

\def\TO{\Longrightarrow}
\def\In{\subseteq}

\def\into{\hookrightarrow}

\def\prefix{\sqsubseteq}

\def\mto{\rightrightarrows}

\def\id{{\rm id}}

\def\dom{{\rm dom}}
\def\range{{\rm range}}

\def\Baire{{\IN^\IN}}

\def\cc{\mathrm c}

\def\Tr{{\rm Tr}}

\newcommand{\SO}[1]{{{\bf\Sigma}^0_{#1}}}

\newcommand{\PO}[1]{{{\bf\Pi}^0_{#1}}}

\newcommand{\pO}[1]{{\Pi^0_{#1}}}

\def\LPO{\text{\rm\sffamily LPO}}
\def\LLPO{\text{\rm\sffamily LLPO}}
\def\WKL{\text{\rm\sffamily WKL}}
\def\BWKL{\text{\rm\sffamily BWKL}}

\def\BCT{\text{\rm\sffamily BCT}}

\def\BWT{\text{\rm\sffamily BWT}}

\def\C{\mbox{\rm\sffamily C}}

\def\LPO{\mbox{\rm\sffamily LPO}}
\def\LLPO{\mbox{\rm\sffamily LLPO}}

\def\MLPO{\mbox{\rm\sffamily MLPO}}

\def\WBWT{\text{\rm\sffamily WBWT}}
\def\SBWT{\text{\rm\sffamily SBWT}}

\def\Low{\text{\rm\sffamily L}}
\def\CL{\text{\rm\sffamily CL}}

\def\J{\text{\rm\sffamily J}}

\def\JIT{\text{\rm\sffamily JIT}}
\def\LBT{\text{\rm\sffamily LBT}}
\def\KPT{\text{\rm\sffamily KPT}}

\def\HYP{\text{\rm\sffamily HYP}}
\def\NON{\text{\rm\sffamily NON}}
\def\DNC{\text{\rm\sffamily DNC}}
\def\ACC{\text{\rm\sffamily ACC}}

\def\MLR{\text{\rm\sffamily MLR}}
\def\WWKL{\text{\rm\sffamily WWKL}}
\def\GEN{\text{\rm\sffamily GEN}}
\def\WGEN{\text{\rm\sffamily WGEN}}

\def\COH{\text{\rm\sffamily COH}}
\def\PA{\text{\rm\sffamily PA}}

\def\leqM{\mathop{\leq_{\mathrm{M}}}}

\def\leqT{\mathop{\leq_{\mathrm{T}}}}
\def\nleqT{\mathop{\not\leq_{\mathrm{T}}}}
\def\equivT{\mathop{\equiv_{\mathrm{T}}}}
\def\nequivT{\mathop{\not\equiv_{\mathrm{T}}}}

\def\nT{\mathop{|_{\mathrm{T}}}}

\def\leqW{\mathop{\leq_{\mathrm{W}}}}
\def\equivW{\mathop{\equiv_{\mathrm{W}}}}

\def\leqSW{\mathop{\leq_{\mathrm{sW}}}}

\def\equivSW{\mathop{\equiv_{\mathrm{sW}}}}

\def\nleqW{\mathop{\not\leq_{\mathrm{W}}}}
\def\nleqSW{\mathop{\not\leq_{\mathrm{sW}}}}

\def\lW{\mathop{<_{\mathrm{W}}}}
\def\lSW{\mathop{<_{\mathrm{sW}}}}
\def\nW{\mathop{|_{\mathrm{W}}}}
\def\nSW{\mathop{|_{\mathrm{sW}}}}

\def\bigtimes{\mathop{\mathsf{X}}}

\def\stars{*_{\rm s}\;\!}

\newcommand{\dash}{\mbox{-}}

\date{\today}

\newtheorem{theorem}{Theorem}[section]
\newtheorem{proposition}[theorem]{Proposition}
\newtheorem{lemma}[theorem]{Lemma}
\newtheorem{fact}[theorem]{Fact}

\newtheorem{corollary}[theorem]{Corollary}

\theoremstyle{definition}
\newtheorem{definition}[theorem]{Definition}

\newtheorem{example}[theorem]{Example}

\newtheorem{question}[theorem]{Question}

\begin{document}

\maketitle

\begin{abstract}
We demonstrate that the Weihrauch lattice can be used to classify 
the uniform computational content of computability-theoretic properties as well as the computational content 
of theorems in one common setting. 
The properties that we study include diagonal non-computability, hyperimmunity, 
complete consistent extensions of Peano arithmetic, 1-genericity, Martin-L\"of randomness,
and cohesiveness.  The theorems that we include in our case study are the
low basis theorem of Jockusch and Soare, the Kleene-Post theorem, and Friedberg's jump inversion theorem.
It turns out that all the aforementioned properties and many theorems in computability theory,
including all theorems that claim the existence of some Turing degree, have very little uniform computational content:
they are located outside of the upper cone of binary choice (also known as LLPO); we call problems with
this property {\em indiscriminative}.
Since practically all theorems from classical analysis whose computational content has been classified
are discriminative, our observation could
yield an explanation for why theorems and results in computability theory typically have very
few direct consequences in other disciplines such as analysis.
A notable exception in our case study is the low basis theorem which is discriminative. This is perhaps why it is
considered to be one of the most applicable theorems in computability theory.
In some cases a bridge between the indiscriminative world and the discriminative world of classical mathematics can be 
established via a suitable residual operation and we demonstrate this in the case of 
the cohesiveness problem and the problem of consistent complete extensions of Peano arithmetic.
Both turn out to be the quotient of two discriminative problems.
\end{abstract}

\setcounter{tocdepth}{1}
\tableofcontents

\section{Introduction}
\label{sec:introduction}

In this paper we start a new line of research with the aim of analyzing the uniform computational content
of theorems and properties of computability theory in the Weihrauch lattice.
This is very much in the spirit of previous research by Simpson \cite{Sim09a}, Downey et al.\ \cite{DGJM11}, Jockusch and Lewis \cite{JL13} and others
who have compared computability-theoretic properties in the Medvedev lattice and, in fact, we can 
import results from this setting (see Lemma~\ref{lem:Medvedev}).

However, the Weihrauch lattice allows a higher degree of uniformity than the Medvedev lattice in the sense that problems that depend on a parameter
can be classified: the objects in the Weihrauch lattice are (multi-valued) functions on Baire space and not just subsets of Baire space. Hence, we can not only
include computability properties in their fully relativized versions into our study but also theorems of computability theory
(in for all-exists form).

In this regard, this paper can be seen as a continuation of the classification of the uniform computational content of theorems from analysis,
which has been started in \cite{GM09,Pau10,BG11,BG11a,BBP12,BGM12,BLP12a,BGH15a,DDH+16}.
Theorems that have been classified in this context include, among others:

\begin{enumerate}
\item the Hahn-Banach theorem \cite{GM09},
\item the intermediate value theorem \cite{BG11a},
\item the (functional analytic) Baire category theorem \cite{BG11a},
\item the Banach inverse mapping theorem \cite{BG11a},
\item the Bolzano-Weierstra\ss{} theorem \cite{BGM12},
\item the Brouwer fixed point theorem \cite{BLP12a},
\item the Nash equilibria existence theorem \cite{Pau10},
\item the Radon-Nykodim theorem \cite{HRW12}.
\end{enumerate}

All these theorems have a certain computational property of combinatorial nature in common: they all compute binary choice---we call a theorem with this property {\em discriminative}. Mathematically speaking, this means that
choice of the two point space $\C_2$ or, equivalent $\LLPO$, is Weihrauch reducible to the problem
in question. 

The most basic part of computability theory is fully constructive
and computable, but as soon as the theory advances to non-trivial existence results it becomes increasingly 
non-constructive. The purpose of this study is to characterize the exact amount of constructivity (which corresponds to
the uniform computational content) of computability theory in terms of the Weihrauch lattice. 
In particular, we study the following theorems from computability theory:

\begin{enumerate}
\item the Kleene-Post theorem $\KPT$,
\item Friedberg's jump inversion theorem $\JIT$,
\item the low basis theorem of Jockusch and Soare $\LBT$.
\end{enumerate}

It turns out that almost all theorems of computability theory that we have studied---with the notable exception of the low basis theorem--- are 
{\em indiscriminative}: they do not even compute binary choice. As a consequence of this, a large part of 
computability theory cannot have any direct implications on more classical mathematics,
such as the part of analysis discussed above. More precisely, no indiscriminative theorem 
computes any discriminative theorem; that is, if $A$ is discriminative and $B$ is indiscriminative, then $A$ is not Weihrauch reducible to $B$. Of course, this does not mean that computability theory is not useful in analysis, 
it only means that some typical theorems in analysis cannot be a {\em direct consequence} (in the sense made precise
by Weihrauch reducibility) of a typical theorem in computability theory.

The one exception in the computability theory that we have considered is the low basis theorem, 
and this may explain why this theorem has been considered as a particularly applicable theorem
within computability theory. For instance, Barry Cooper writes of the low basis theorem \cite[Page 330]{Coo04}: 
\begin{quotation}
``Here is our most useful basis, proved by Jockusch and Soare around 1972.
You often come across applications of it in the most unexpected places." 
\end{quotation}
We identify one common topological reason behind the fact that large parts of computability theory are
indiscriminative: the corresponding theorems are densely realized in the sense that for-all-exists 
statements of the form 
\[(\forall x\in X)(\exists y\in Y)\;P(x,y)\] 
have multi-valued realizers (i.e., multi-valued Skolem functions)
whose image is dense in $Y$. 
In Section~\ref{sec:indiscriminative} we give a more precise definition of this property, which we call {\em densely realized}.
Sometimes this is a specific feature of $P$, but often a statement is densely realized because of the mere
choice of $Y$. For instance, a Turing degree is naturally represented by one of its members $p\in\IN^\IN$.
By $\DD$ we denote the set of Turing degrees and we use the symbol $\leqT$ for Turing reducibility on degrees.
Since Turing degrees are invariant under finite
modifications of their representatives, we immediately obtain that {\em any} existence theorem of the
form $(\forall x\in X)(\exists {\mathbf d}\in\DD)\;P(x,{\mathbf d})$, which claims the existence of a Turing degree, is automatically
densely realized and hence indiscriminative. This explains why large parts of computability are indiscriminative. 

The particular computability-theoretic properties that we include in our study are the following (of which only $\DNC$ is discriminative):

\begin{enumerate}
\item diagonal non-computability $\DNC$,
\item complete consistent extensions of Peano arithmetic $\PA$,
\item hyperimmunity $\HYP$,
\item Martin-L\"of randomness $\MLR$,
\item $1$--genericity $1\dash\GEN$,
\item cohesiveness $\COH$.
\end{enumerate}

In fact, one can also interpret all these properties as existence theorems for corresponding objects. 
For instance, $\MLR$ can be seen as the existence
of a Martin-L\"of random point $q$ relative to some given $p$, so $\MLR(p)$ is the set of 
all Martin-L\"of randoms relative to $p$. In this sense we consider fully relativized versions of these
properties in the Weihrauch lattice.

$\COH$ has already been studied from a uniform perspective, in the context
of combinatorial principles, by Dorais et al.\ \cite{DDH+16} and Hirschfeldt \cite{Hir15}.
$\MLR$ has been considered from a uniform perspective in \cite{BGH15} and \cite{BP17}. 
$\DNC$ has independently been studied by Higuchi and Kihara \cite{HK14a}.

For a quick survey of our results the reader might wish to consult the diagram in 
Figure~\ref{fig:diagram-DNC} on Page~\pageref{fig:diagram-DNC} and, in particular,
the final diagram in Figure~\ref{fig:diagram-Full} on Page~\pageref{fig:diagram-Full}.

In Section~\ref{sec:preliminaries} we present some basic facts on the Weihrauch lattice
that are included in order to keep this paper self-contained. 
In Section~\ref{sec:choice} we include a preliminary discussion on some versions
of choice and omniscience that we need throughout this paper.
In Section~\ref{sec:indiscriminative} we introduce the notion of an indiscriminative problem
or theorem and we prove some basic facts about them. 
In Section~\ref{sec:DNC} we discuss diagonal non-computability, and in Section~\ref{sec:PA} we study
the closely related problem of complete extensions of Peano arithmetic. Hereafter, whenever we refer to complete extensions of Peano arithmetic we mean complete and consistent extensions. Among other results, we 
prove the characterization
\[\PA\equivW(\C_\IN'\to\WKL).\]
This shows that $\PA$ can be seen as a quotient of weak K\H{o}nig's lemma $\WKL$ and $\C_\IN'$ (which denotes the jump of choice on
the natural numbers). Results like this show how an indiscriminative problem can be used to characterize
the relationship between two discriminative problems and hence establish a connection between analysis 
and computability theory. 
In Section~\ref{sec:MLR-WWKL} we consider Martin-L\"of randomness and different uniform versions of
weak weak K\H{o}nig's lemma (denoted by $\WWKL$) in relation to the previously mentioned problems. 
The relevant versions of $\WWKL$ are taken from \cite{BGH15a}.
The results achieved in this context are summarized in the diagram in Figure~\ref{fig:diagram-DNC} on Page~\pageref{fig:diagram-DNC}.
In Section~\ref{sec:LBT} we classify the uniform computational content of the low basis theorem in 
relation to other known problems related to lowness and weak K\H{o}nig's lemma. 
In Section~\ref{sec:HYP} we continue with a study of the hyperimmunity problem.
In Sections~\ref{sec:KPT} and \ref{sec:JIT} we study the Kleene-Post theorem and the jump inversion theorem
of Friedberg, respectively. The well-known fact that van Lambalgen's theorem can be used to 
create a pair of incomparable Turing degrees yields a uniform reduction of the Kleene-Post theorem to $\MLR$. 
Likewise a theorem of Yu allows a reduction of $\KPT$ to $1\dash\GEN$.
The jump inversion theorem appears to be the only current example of a theorem that is
continuous but not computable. 
Sections~\ref{sec:COH} and \ref{sec:COH-deg} are devoted to the cohesiveness problem
and Section~\ref{sec:BWT} to closely related weak versions of the Bolzano-Weierstra\ss{} theorem.
In this section we borrow some ideas of the proof-theoretic study of cohesiveness and the Bolzano-Weierstra\ss{} theorem
in \cite{Kre11}. One of the main results on the cohesiveness problem is that it can be characterized with the help
of the limit operation $\lim$  and the jump of weak weak K\H{o}nig's lemma by
\[\COH\equivW(\lim\to\WWKL').\]
This is another instance of the phenomenon that an indiscriminative problem appears as a quotient
between discriminative problems. 
The results of this paper are extended by further results on the uniform computational content 
of different versions of the Baire category theorem and $1$--genericity that are presented in \cite{BHK16}.
These results are also included in the diagram in Figure~\ref{fig:diagram-Full} on Page~\pageref{fig:diagram-Full}.

\section{Preliminaries}
\label{sec:preliminaries}

In this section we give a brief introduction into the Weihrauch lattice and we provide some basic notions from
probability theory. 

\subsection*{Pairing Functions}

We are going to use some standard pairing functions in the following that we briefly summarize.
By $\IN:=\{0,1,2,...\}$ we denote the set of {\em natural numbers}.
As usual, we denote by 
$\langle n,k\rangle :=\frac{1}{2}(n+k+1)(n+k)+k$
the Cantor pair of two natural numbers $n,k\in\IN$.
We define the pairing $\langle p,q\rangle\in\IN^\IN$ of two sequences $p,q\in\IN^\IN$ by 
$\langle p,q\rangle(n):=p(k)$ when $n=2k$ is even and by $\langle p,q\rangle(n)=q(k)$ when $n=2k+1$ is odd.
By $\langle k,p\rangle:=kp\in\IN^\IN$ 
we denote the natural pairing of a number $k\in\IN$ with a sequence $p\in\IN^\IN$.
By $\IN\times2^\IN$ we denote the set of sequences that have a natural number as first component and a 0 or 1 for every other component.
We also define a pairing function
$\langle p_0,p_1\rangle:=\langle\langle p_0(0),p_1(0)\rangle,\langle \overline{p_0},\overline{p_1}\rangle\rangle$,
for $p_0,p_1\in\IN\times2^\IN$, where $\overline{p_i}(n)=p_i(n+1)$.
Finally, we use the pairing $\langle p_0,p_1,p_2,...\rangle\in\IN^\IN$ for $p_i\in\IN^\IN$ defined by
$\langle p_0,p_1,p_2,...\rangle\langle i,j\rangle:=p_i(j)$.

\subsection*{The Weihrauch Lattice}

The original definition of Weihrauch reducibility is due to Klaus Weihrauch
and has been studied for many years (see \cite{Ste89,Wei92a,Wei92c,Her96,Bra99,Bra05}).
More recently it has been noticed that a certain variant of this reducibility yields
a lattice that is very suitable for the classification of the computational content of mathematical theorems
(see  \cite{GM09,Pau10,Pau10a,BG11,BG11a,BBP12,BGM12}). The basic reference for all notions
from computable analysis is Weihrauch's textbook \cite{Wei00}.
The Weihrauch lattice is a lattice of multi-valued functions on represented
spaces. 

A {\em representation} $\delta$ of a set $X$ is just a surjective partial
map $\delta:\In\IN^\IN\to X$. In this situation we call $(X,\delta)$ a {\em represented space}.
In general we use the symbol ``$\In$'' to indicate that a function is potentially partial.
We work with partial multi-valued functions $f:\In X\mto Y$ where $f(x)\In Y$ denotes the set of possible
values upon input $x\in\dom(f)$. If $f$ is single-valued, then for the sake of simplicity we identify $f(x)$ with its unique inhabitant.
We denote the {\em composition} of
two (multi-valued) functions $f:\In X\mto Y$ and $g:\In Y\mto Z$ either by $g\circ f$ or by $gf$.
It is defined by 
\[g\circ f(x):=\{z\in Z:(\exists y\in Y)(z\in g(y)\mbox{ and }y\in f(x))\},\]
where $\dom(g\circ f):=\{x\in X:f(x)\In\dom(g)\}$.
Using represented spaces we can define the concept of a realizer. 

\begin{definition}[Realizer]
Let $f : \In (X, \delta_X) \mto (Y, \delta_Y)$ be a multi-valued function on represented spaces.
A function $F:\In\IN^\IN\to\IN^\IN$ is called a {\em realizer} of $f$, in symbols $F\vdash f$, if
$\delta_YF(p)\in f\delta_X(p)$ for all $p\in\dom(f\delta_X)$.
\end{definition}

Realizers allow us to transfer the notions of computability
and continuity, and other notions available for Baire space, to any represented space;
a function between represented spaces will be called {\em computable} if it has a computable realizer, etc.
Now we can define Weihrauch reducibility.

\begin{definition}[Weihrauch reducibility]
Let $f,g$ be multi-valued functions on represented spaces. 
Then $f$ is said to be {\em Weihrauch reducible} to $g$, in symbols $f\leqW g$, if there are computable
functions $K,H:\In\IN^\IN\to\IN^\IN$ such that $H\langle \id, GK \rangle \vdash f$ for all $G \vdash g$.
Moreover, $f$ is said to be {\em strongly Weihrauch reducible} to $g$, in symbols $f\leqSW g$,
if the analogous condition holds with the property $HGK\vdash f$ in place of  $H\langle \id, GK \rangle \vdash f$.
\end{definition}

The difference between ordinary and strong Weihrauch reducibility is that the ``output modifier'' $H$ has
direct access to the original input in the case of ordinary Weihrauch reducibility but not in the case of strong Weihrauch reducibility. 
There are algebraic and other reasons to consider ordinary Weihrauch reducibility as the more natural variant. 
For instance, one can characterize the reduction $f\leqW g$ as follows: $f\leqW g$ holds if and only if a Turing machine can compute $f$ in such a way that
it evaluates the ``oracle'' $g$ exactly on one (usually infinite) input during the course of its computation (see \cite[Theorem~7.2]{TW11}).
We will use the strong variant $\leqSW$ of Weihrauch reducibility mostly for technical purposes; for instance,
it is better suited to study jumps, since jumps are monotone with respect to strong reductions but not in general for ordinary reductions.

We note that the relations $\leqW$, $\leqSW$ and $\vdash$ implicitly refer to the underlying representations, which
we will only mention explicitly if necessary. It is known that these relations only depend on the underlying equivalence
classes of representations and not on the specific representatives (see Lemma~2.11 in \cite{BG11}).
The relations $\leqW$ and $\leqSW$ are reflexive and transitive; thus they induce corresponding partial orders on the sets of 
their equivalence classes---we refer to these as {\em Weihrauch degrees} and {\em strong Weihrauch degrees}, respectively.
These partial orders will also be denoted by $\leqW$ and $\leqSW$. The induced lattice and semi-lattice, respectively, are distributive
(for~details~see~\cite{Pau10a}~and~\cite{BG11}).
We use $\equivW$ and $\equivSW$ to denote the respective equivalences regarding $\leqW$ and $\leqSW$, 
by $\lW$ and $\lSW$ we denote strict reducibility and by $\nW,\nSW$ we denote incomparability in the respective senses.

\subsection*{The Algebraic Structure}

The partially ordered structures induced by the two variants of Weihrauch reducibility are equipped with a number of useful algebraic operations that we summarize in the next definition.
We use $X\times Y$ to denote the ordinary set-theoretic {\em product}, $X\sqcup Y:=(\{0\}\times X)\cup(\{1\}\times Y)$ 
to denote {\em disjoint sums} or {\em coproducts}, and $\bigsqcup_{i=0}^\infty X_i:=\bigcup_{i=0}^\infty(\{i\}\times X_i)$ to denote the 
{\em infinite coproduct}. By $X^i$ we denote the $i$--fold product of a set $X$ with itself, where $X^0=\{()\}$ is some canonical singleton.
By $X^*:=\bigsqcup_{i=0}^\infty X^i$ we denote the set of all {\em finite sequences over $X$}
and by $X^\IN$ the set of all {\em infinite sequences over $X$}. 
All these constructions have parallel canonical constructions on representations. The corresponding representations
are denoted by: $[\delta_X,\delta_Y]$ for the product of $(X,\delta_X)$ and $(Y,\delta_Y)$; $\delta_X^n$ for the $n$-fold product of $(X,\delta_X)$ with itself, where $n\in\IN$ and $\delta_X^0$ is a representation of the one-point set $\{()\}=\{\varepsilon\}$; $\delta_X\sqcup\delta_Y$ for the coproduct; $\delta^*_X$ for $X^*$; and $\delta_X^\IN$ for $X^\IN$. For instance, $(\delta_X\sqcup\delta_Y)$ can be defined by $(\delta_X\sqcup\delta_Y)\langle n,p\rangle:=(0,\delta_X(p))$
if $n=0$ and $(\delta_X\sqcup\delta_Y)\langle n,p\rangle:=(1,\delta_Y(p))$ otherwise.
Likewise, $\delta^*_X\langle n,p\rangle:=(n,\delta_X^n(p))$.
See \cite{Wei00} or \cite{BG11,Pau10a,BBP12} for details of the definitions of the other representations. We will always assume that these canonical representations
are used, if not mentioned otherwise. 

\begin{definition}[Algebraic operations]
\label{def:algebraic-operations}
Let $f:\In X\mto Y$ and $g:\In Z\mto W$ be multi-valued functions. Then we define
the following operations:
\begin{enumerate}
\itemsep 0.2cm
\item $f\times g:\In X\times Z\mto Y\times W, (f\times g)(x,z):=f(x)\times g(z)$ \hfill (product)
\item $f\sqcap g:\In X\times Z\mto Y\sqcup W, (f\sqcap g)(x,z):=f(x)\sqcup g(z)$ \hfill (sum)
\item $f\sqcup g:\In X\sqcup Z\mto Y\sqcup W$, with $(f\sqcup g)(0,x):=\{0\}\times f(x)$ and\\
        $(f\sqcup g)(1,z):=\{1\}\times g(z)$ \hfill (coproduct)
\item $f^*:\In X^*\mto Y^*,f^*(i,x):=\{i\}\times f^i(x)$ \hfill (finite parallelization)
\item $\widehat{f}:\In X^\IN\mto Y^\IN,\widehat{f}(x_n):=\bigtimes\limits_{i\in\IN} f(x_i)$ \hfill (parallelization)
\end{enumerate}
\end{definition}

In this definition and in general we denote by $f^i:\In X^i\mto Y^i$ the $i$--th fold product
of the multi-valued map $f$ with itself ($f^0$ is the constant function on the canonical singleton).
It is known that $f\sqcap g$ is the {\em infimum} of $f$ and $g$ with respect to both strong and
ordinary Weihrauch reducibility (see \cite{BG11}, where this operation was denoted by $\oplus$).
Correspondingly, $f\sqcup g$ is known to be the {\em supremum} of $f$ and $g$ with respect to ordinary Weihrauch reducibility $\leqW$ (see \cite{Pau10a}).
This turns the partially ordered structure of Weihrauch degrees (induced by $\leqW$) into a lattice,
which we call the {\em Weihrauch lattice}.
The two operations $f\mapsto\widehat{f}$ and $f\mapsto f^*$ are known to be closure operators
in this lattice (see \cite{BG11,Pau10a}).

There is some useful terminology related to these algebraic operations. 
We say that $f$ is a {\em cylinder} if $f\equivSW\id\times f$ where $\id:\Baire\to\Baire$ here, and hereafter,
denotes the identity on Baire space. 
For a cylinder $f$ and any $g$ we have that $g\leqW f$ is equivalent to $g\leqSW f$ (see \cite{BG11}).
We say that $f$ is {\em idempotent} if $f\equivW f\times f$ and {\em strongly idempotent} if $f\equivSW f\times f$. 
We say that a multi-valued function $f$ on represented spaces is {\em pointed} if it has a computable
point in its domain. This property is equivalent to $\id\leqW f$ and hence we call $f$ {\em strongly pointed} if
$\id\leqSW f$ holds. For pointed $f$ and $g$ we obtain $f\sqcup g\leqSW f\times g$. 
The properties of pointedness and idempotency are both preserved under
equivalence and hence they can be considered as properties of the respective degrees.
For a pointed $f$, the finite parallelization $f^*$ can also be considered as {\em idempotent closure} since idempotency is equivalent to $f\equivW f^*$ in this case.
We call $f$ {\em parallelizable} if $f\equivW\widehat{f}$; it is easy to see that $\widehat{f}$ is always idempotent.
Analogously, we call $f$ {\em strongly parallelizable} if $f\equivSW\widehat{f}$.

More generally, we define {\em countable coproducts} $\bigsqcup_{i\in\IN} f_i:\In\bigsqcup_{i\in\IN} X_i\mto\bigsqcup_{i\in\IN} Y_i$ 
for a sequence $(f_i)$ of multi-valued functions $f_i:\In X_i\mto Y_i$ on represented spaces with the operation given by $(\bigsqcup_{i\in\IN} f_i)(i,u):=\{i\}\times f_i(u)$. 
Using this notation we obtain $f^*=\bigsqcup_{i\in\IN} f^i$. 
We can also define a {\em countable sum} $\bigsqcap_{i\in\IN} f_i:\In \bigtimes_{i\in\IN} X_i\mto\bigsqcup_{i\in\IN} Y_i$, by $\left(\bigsqcap_{i\in\IN} f_i\right)(x_i)_i:=\bigsqcup_{i\in\IN} f_i(x_i)$.

One should note however, that $\bigsqcap$ and $\bigsqcup$ do not provide infima and
suprema of sequences. By a result of Higuchi and Pauly \cite[Proposition~3.15]{HP13} the Weihrauch lattice has no non-trivial suprema (i.e., a sequence
$(s_n)$ has a supremum $s$ if and only if $s$ is already the supremum of a finite prefix of the sequence $(s_n)_n$)
and likewise by \cite[Corollary~3.18]{HP13} the pointed Weihrauch degrees do not have non-trivial infima. 
In particular, the Weihrauch lattice is not complete.\footnote{We note, however, that for the continuous variant of Weihrauch reducibility
the objects $\bigsqcup_{n\in\IN}f_n$ and $\bigsqcap_{n\in\IN}f_n$ are suprema and infima of the sequence $(f_n)_n$, respectively, and the corresponding
continuous version of the Weihrauch lattice is actually countably complete (see \cite{HP13}).}

\subsection*{Compositional Products and Implications}

While the Weihrauch lattice is not complete, some suprema and some infima exist in general.
The following result was proved by the first author and Pauly in \cite{BP16} and ensures the existence of certain important maxima and minima.

\begin{proposition}[Compositional products and implication]
Let $f,g$ be multi-valued functions on represented spaces. Then the following Weihrauch degrees exist:
\begin{enumerate}
\itemsep 0.2cm
\item $f *g:=\max\{f_0\circ g_0:f_0\leqW f\mbox{ and }g_0\leqW g\}$ \hfill (compositional product)
\item $f\to g:=\min\{h:g\leqW f* h\}$ \hfill (implication)
\end{enumerate}
\end{proposition}

Here $f*g$ is defined over all $f_0\leqW f$ and $g_0\leqW g$ which can actually be composed (i.e., the target space of $g_0$ and the source space of $f_0$ have to coincide).
In this way $f*g$ characterizes the most complicated Weihrauch degree that can be obtained by first performing a computation with the help of $g$ and then another one with the help of $f$.
Since $f*g$ is a maximum in the Weihrauch lattice, we can consider $f*g$ as some fixed representative of the corresponding degree.
It is easy to see that $f\times g\leqW f*g$ holds. 
The compositional products were originally introduced in \cite{BGM12}.
The implication $f\to g$ represents the weakest oracle which is needed in advance of $f$ in order to compute $g$ and it was introduced in \cite{BP16}.

We can also define the {\em strong compositional product} by 
\[f\stars g:=\max\{f_0\circ g_0:f_0\leqSW f\mbox{ and }g_0\leqSW g\},\]
where the maximum is taken with respect to $\leqSW$, 
but we neither claim that it exists in general. However, for cylinders $f,g$ this is the case.

\begin{lemma}[Strong compositional product]
\label{lem:strong-composition}
If $f,g$ are cylinders, then $f\stars g$ exists and is also a cylinder.
\end{lemma}
\begin{proof}
Let us consider $M:=\{f_0\circ g_0:f_0\leqSW f,g_0\leqSW g\}$. We need to show that $f\stars g\equivSW\max_{\leqSW}(M)$ exists.
Since $f,g$ are cylinders, we can replace both strong Weihrauch reductions in the definition of $M$ by ordinary Weihrauch reductions.
In \cite{BP16} a specific problem $f\star g\in M$ was defined and it was proved in \cite[Corollary~18]{BP16} that $f\star g\equivW\max_{\leqW}(M)$ holds.
This implies that $h\leqW f\star g$ for all $h\in M$.
Moreover, $f\star g$ is a cylinder by \cite[Lemma~17]{BP16} and hence $h\leqSW f\star g$ follows for all $h\in M$. 
Altogether, this proves the claim.
\end{proof}

We can conclude that $\stars$ is associative whenever it exists, which basically follows from associativity of the usual composition.

\begin{lemma}[Associativity]
\label{lem:associative}
$f\stars(g\stars h)\equivSW(f\stars g)\stars h$ holds for all $f,g,h$ such that all involved strong compositional products exist.
\end{lemma}
\begin{proof}
We prove $f\stars(g\stars h)\leqSW(f\stars g)\stars h$. The inverse reduction can be proved analogously.
Let $f_0\leqSW f$ and $g_0\leqSW g\stars h$. Without loss of generality we can assume the functions 
are of type $f_0,g_0,f,g,h:\In\IN^\IN\mto\IN^\IN$. It suffices to show that $f_0\circ g_0\leqSW(f\stars g)\stars h$.
The assumption and the fact that $g\stars h$ exists together imply that there are computable $H_0,K_0,H_1,G_1,K_1$ such that
$\emptyset\not=H_0fK_0(p)\In f_0(p)$ for all $p\in\dom(f_0)$ and $\emptyset\not=H_1gG_1hK_1(p)\In g_0(p)$ for all $p\in\dom(g_0)$. Hence 
\[f_0\circ g_0\leqSW H_0fK_0H_1gG_1hK_1\leqSW (f\stars g)\stars h.\qedhere\]
\end{proof}

\subsection*{Jumps}

In \cite{BGM12}  {\em jumps} or {\em derivatives} of multi-valued functions on represented spaces were introduced.
The {\em jump} $f':\In (X,\delta_X')\mto (Y,\delta_Y)$ of a multi-valued function $f:\In (X,\delta_X)\mto (Y,\delta_Y)$ on represented
spaces is obtained by replacing the input representation $\delta_X$ by its jump $\delta'_X:=\delta_X\circ\lim$, where
\[\lim:\In\IN^\IN\to\IN^\IN,\langle p_0,p_1,p_2,...\rangle\mapsto\lim_{n\to\infty}p_n\] 
is the limit operation on Baire space $\IN^\IN$ with respect to the product topology on $\IN^\IN$. It follows that $f'\equivSW f\stars\lim$
(see \cite[Corollary~5.16]{BGM12}). By $f^{(n)}$ we denote the $n$--fold jump. 
A $\delta_X'$--name $p$ of a point $x\in X$ is a sequence that converges to a $\delta_X$--name of $x$.
This means that a $\delta_X'$--name typically contains significantly less accessible information on $x$ than a $\delta_X$--name. 
Hence $f'$ is typically harder to compute than $f$, since less input information is available for $f'$.

The jump operation $f\mapsto f'$ plays a similar role in the Weihrauch lattice as the Turing jump operation
does in the Turing semi-lattice. In a certain sense $f'$ is a version of $f$ on the ``next higher'' level of complexity
(this can be made precise using the Borel hierarchy \cite{BGM12}).
It was proved in \cite{BGM12} that the jump operation $f\mapsto f'$ is monotone with respect to strong Weihrauch 
reducibility $\leqSW$ but not with respect to ordinary Weihrauch reducibility $\leqW$. This is another reason
why it is beneficial to extend the study of the Weihrauch degrees to the strong Weihrauch degrees. 

In general we use the notation $\lim_X:\In X^\IN\to X$ for the limit map of a metric space $X$.

\subsection*{Closed Choice}

A particularly useful multi-valued function in the Weihrauch lattice is closed choice $\C_X$ (see \cite{GM09,BG11,BG11a,BBP12}); it is known that many notions of computability can be calibrated using closed choice for a suitable computable metric 
space $X$. We recall that a {\em computable metric space} $X$ is a separable metric space with a given dense sequence 
such that the distance function is computable
on this sequence \cite{Wei00}. If $(B_n)_n$ is a standard enumeration of the open balls of $X$ with center
from the dense subset and rational radius (including the empty ball), then an open subset $U\In X$ is called {\em c.e.\ open}
if there is a computable $p\in\IN^\IN$ such that $U=\bigcup_{i=0}^\infty B_{p(i)}$.
Correspondingly, a closed subset $A\In X$ is called {\em co-c.e.\ closed} if $X\setminus A$ is c.e.\ open \cite{BP03}.
For subsets $A$ of the natural numbers 
this leads to the usual notion of c.e.\ and co-c.e.\ sets.
The co-c.e.\ closed subsets $A\In\IN^\IN$ of Baire space are exactly the usual $\pO{1}$--classes.

By $\AA_-(X)$ we denote the set of closed subsets of $X$ with the representation $\psi_-$ given by
\[\psi_-(p):=X\setminus\bigcup_{i=0}^\infty B_{p(i)}.\]
We are now prepared to define closed choice.

\begin{definition}[Closed Choice]
Let $X$ be a computable metric space. The {\em closed choice} problem 
of the space $X$ is defined by
\[\C_X:\In\AA_-(X)\mto X,A\mapsto A\]
with $\dom(\C_X):=\{A\in\AA_-(X):A\not=\emptyset\}$.
\end{definition}

Intuitively, $\C_X$ takes as input a non-empty closed set in negative description (i.e., given by $\psi_-$) 
and it produces an arbitrary point of this set as output.
Hence, $A\mapsto A$ means that the multi-valued map $\C_X$ maps
the input $A\in\AA_-(X)$ to the set $A\In X$ as a set of possible outputs.
We mention some classes of functions that can be characterized
by closed choice. The following results have mostly been proved in \cite{BBP12}.

\begin{proposition}
\label{prop:classes}
Let $f$ be a multi-valued function on represented spaces. Then
\begin{enumerate}
\item $f\leqW\C_1\iff f$ is computable,
\item $f\leqW\C_\IN\iff f$ is computable with finitely many mind changes,
\item $f\leqW\C_{2^\IN}\iff f$ is non-deterministically computable,
\item $f\leqW\C_{\IN^\IN}\iff f$ is effectively Borel measurable.
\end{enumerate}
In case (4) we have to assume that $f:X\to Y$ is single-valued and defined
on computable complete metric spaces $X,Y$.
\end{proposition}

Here, and in general, we identify each natural number $n\in\IN$ with the corresponding finite subset $n=\{0,1,...,n-1\}$.
The problem $\C_0$,  closed choice for the empty set $0=\emptyset$, is the bottom element of the Weihrauch lattice.
In \cite{BGM12} the jumps $\C_X'\equivW\CL_X$ for computable metric spaces $X$ were characterized using
the cluster point problem $\CL_X$ of $X$, which is defined by
\[\CL_X:\In X^\IN\mto X,(x_n)_n\mapsto\{x\in X:x\mbox{ is a cluster point of $(x_n)_n$}\}\]
where $\dom(\CL_X)$ is the set of all sequences that have a cluster point.

\section{Choice and Omniscience}
\label{sec:choice}

In this paper we will make some essential use of a number of choice principles that we briefly 
discuss in this section.
Firstly, $\C_2$ is the problem of choosing a point in a non-empty subset $A\In\{0,1\}$ given by negative information. This information basically is a sequence that might contain no information
at all or it might eventually contain the information that one of the points $0$ or $1$ is not included in $A$.
It has been noticed \cite[Example~3.2]{BBP12} that $\C_2$ is equivalent to the {\em lesser limited principle of omniscience} $\LLPO$ 
as it is used in constructive analysis \cite{BB85}.

\begin{fact} 
\label{fact:LLPO}
$\C_2\equivSW\LLPO$.
\end{fact}

More generally, one obtains $\C_n\equivSW\MLPO_n$ for all $n\geq2$ for the generalizations $\MLPO_n$ of $\LLPO$
that have been introduced by Weihrauch in \cite{Wei92c} (where ``$\text{\rm\sffamily M}$'' stands for {\em more} omniscient).
We are interested in the following variant of choice, which we call {\em all or co-unique choice}.\footnote{The name {\em all or co-unique choice} is motivated by
{\em all or unique choice}, which is related to solving linear equations and Nash equilibria,
see \cite{Pau11} or \cite{BGH15a}.}

\begin{definition}[All or co-unique choice]
Let $X$ be a represented space.
By $\ACC_X:\In\AA_-(X)\mto X,A\mapsto A$ we denote the {\em all or co-unique choice operation}
with $\dom(\ACC_X):=\{A\in\AA_-(X):|X\setminus A|\leq 1$ and $|A|>0\}$.
\end{definition}

Here $|A|$ denotes the cardinality of the set $A$. In other words, all or co-unique choice $\ACC_X$ is the problem
of choosing a point in a set $A\In X$ from which at most one element of $X$ is missing. 
It is easy to see that $\ACC_X$ is computable if $X$ has a non-isolated computable point: a non-isolated point cannot be the only point of $X$ that is missing in a closed set $A\In X$ and hence
a computable realizer can just choose that computable point as a solution.
In particular, $\ACC_X$ is computable for perfect computable metric spaces $X$.
This is the reason that, in contrast to $\C_X$, the problem $\ACC_X$ is mostly of interest
for spaces $X\In\IN$.

Again, the principle $\ACC_n$ has been studied before in form of an omniscience principle $\LLPO_n$ by Weihrauch \cite{Wei92c} and
it can be defined as a multi-valued function by
\[\LLPO_n:\In\IN^\IN\mto\IN,\langle p_1,...,p_n\rangle\mapsto\{i\in\IN:p_i=\widehat{0}\}\]
where $\dom(\LLPO_n)$ is the set of those $\langle p_1,...,p_n\rangle$ such that $p_i\not=\widehat{0}$ 
for at most one $i=1,...,k$ and $\widehat{n}\in\IN^\IN$ denotes the constant sequence with value $n\in\IN$.
With this definition we obtain $\LLPO_2=\LLPO$ and the following fact is obvious.

\begin{fact} 
\label{fact:ACC-LLPO}
$\ACC_n\equivSW\LLPO_n$ for all $n\geq 2$.
\end{fact}

The problem $\ACC_n$ also appeared under the name $\C_{n-1,n}$ in the proof of \cite[Theorem~10.1]{BGH15a}
from which we can conclude $\ACC_{n+1}\lW\ACC_n$ for all $n\geq 2$. 
However, this had already been proved by Weihrauch
and with \cite[Theorems~4.3 and 5.4]{Wei92c} we obtain the following result.

\begin{fact}[Weihrauch 1992]
\label{fact:ACC-chain}
For all $n>2$ we obtain
\[\ACC_\IN\lW\ACC_{n+1}\lW\ACC_{n}\lW\ACC_2=\C_2\lW\C_n\lW\C_{n+1}\lW\C_\IN.\]
An analogous result holds true with $\lSW$ instead of $\lW$.
\end{fact}

We note that $\ACC_\IN\leqSW\ACC_{n+1}\leqSW\ACC_n$ holds for all $n\geq2$ and that the condition $\ACC_{n+1}\lW\ACC_{n}$ for all $n\geq2$
implies the strictness of the reduction $\ACC_\IN\lW\ACC_{n+1}$.
We note that the $\LLPO_n$ hierarchy has recently also been separated over IZF+DC~\cite{HL16}.

\section{Indiscriminative Theorems}
\label{sec:indiscriminative}

The purpose of this section is to study a class of multi-valued functions and theorems
with very little uniform content. We will call these functions and theorems {\em indiscriminative}
since they cannot even be utilized to make binary choices.
In the Weihrauch lattice binary choice $\C_2$ represents the ability of making binary choices
and the much weaker principle $\ACC_\IN$ represents the ability of choosing one of countably 
many objects in a setting where at most one object is forbidden.

\begin{definition}[Discriminative degrees]
A multi-valued function $f$ on represented spaces is called {\em discriminative}
if $\C_2\leqW f$. Otherwise, it is called {\em indiscriminative}.
We call $f$ {\em $\omega$--discriminative} if $\ACC_\IN\leqW f$ and {\em $\omega$--indiscriminative} otherwise.
\end{definition}

It follows directly from Fact~\ref{fact:ACC-chain} that every discriminative multi-valued function is also $\omega$--discriminative
and every $\omega$--indiscriminative multi-valued function is also indiscriminative. 
All non-constructive theorems from classical analysis that have been
classified in the Weihrauch lattice so far are discriminative (see for instance \cite{BG11a,BGM12}).
In all cases that we will encounter here indiscriminativity has a common reason; namely, that the corresponding theorem has so many solutions that any finite
portion of a particular solution does not carry any useful information. 
This idea is made precise in the following definition.

\begin{definition}[Dense realization]
Let $f:\In X\mto Y$ be a multi-valued map on represented spaces $(X,\delta_X)$ and $(Y,\delta_Y)$.
Then $f$ is called {\em densely realized} if
$\delta_Y^{-1}\circ f\circ \delta_X(p)$ is dense in $\dom(\delta_Y)$
for every $p\in\dom(f\delta_X)$.  
\end{definition}

It is easy to see that every densely realized $f$
is $\omega$--indiscriminative, and further, that no such $f$ can be strongly pointed or a cylinder.
In the next proposition we formulate an even stronger property. We say that $f$ {\em strongly bounds} $g$ if $g\leqSW f$.

\begin{proposition}[Dense realization]
\label{prop:dense}
Let $f$ be a multi-valued function on represented spaces that is densely realized. 
Then $f$ is $\omega$--indiscriminative and does not strongly bound any single-valued, non-constant function on Baire space. 
In particular, $f$ is not strongly pointed, does not strongly bound any cylinder,
and is not a cylinder itself.
\end{proposition}
\begin{proof}
Let $f$ be densely realized and assume that $\ACC_\IN\leqW f$.
Then there are computable functions $H,K$ such that
$H\langle\id,FK\rangle\vdash\ACC_\IN$ whenever $F\vdash f$.
Consider a name $p$ of $\IN$ and let $G\vdash f$
be some realizer of $f$.
Let us assume $H\langle p,GK(p)\rangle=i\in\IN$.
By the continuity of $H$ there are some words $w\prefix p$ and $v\prefix GK(p)$
such that $H\langle w\IN^\IN,v\IN^\IN\rangle=\{i\}$.
Then we consider some name $q$ of $\IN\setminus\{i\}$ with $w\prefix q$.
Since $f$ is densely realized, there is some realizer $F\vdash f$
such that $v\prefix FK(q)$. Now we obtain $H\langle q,FK(q)\rangle=i$,
which contradicts $H\langle q,FK(q)\rangle\in\IN\setminus\{i\}=\ACC_\IN(\IN\setminus\{i\})$. 
Hence $\ACC_\IN\nleqW f$ and $f$ is $\omega$--indiscriminative.

Now suppose that $g\leqSW f$ where $g:\In\IN^\IN\to\IN^\IN$ is a non-constant function. 
Then there are computable $H,K$ such that $HFK=g$ whenever $F\vdash f$. 
Let $F_0\vdash f$ and let $p,q\in\dom(g)$ with $g(p)\not=g(q)$.
Then there is some prefix $v\prefix g(p)$ such that $v\not\prefix g(q)$.
Then $v\prefix HF_0K(p)$ and by continuity of $H$ there is some finite prefix $w\prefix F_0K(p)$
such that $H(w\IN^\IN)\In v\IN^\IN$. Since $f$ is densely realized, there is some
realizer $F\vdash f$ with $w\prefix FK(q)$ and hence $v\prefix HFK(q)$, which contradicts $HFK(q)=g(q)$. Hence $g\nleqSW f$.

In particular, we obtain $\id\nleqSW f$ and so $f$ is not strongly pointed.
Since $\id\leqSW g$ holds for any cylinder $g$, it follows that $g\leqSW f$ cannot hold for a cylinder $g$.
In particular, $f$ is not a cylinder.
\end{proof}

The first part of the statement of Proposition~\ref{prop:dense} was strengthened in \cite[Proposition~55]{BP16},
where it was shown that $g\leqW f$ implies that $g$ is computable for every $g:\In X\mto \IN$.
Sometimes a multi-valued function is densely realized just due to the mere type of its output.
This is the case if the output representation is densely realized itself in the following sense.

\begin{definition}[Densely realized]
A represented space $(Y,\delta)$ is called {\em densely realized} if $\delta^{-1}(y)$ is dense
in $\dom(\delta)$ for each $y\in Y$.
\end{definition}

It is clear that the final topology\footnote{We recall that a partial map $f:\In X\to Y$ from a topological space $X$ to a set $Y$
induces a topology on $Y$, which is called the {\em final topology}. A set $V\In Y$ belongs to this topology if and only if $f^{-1}(V)$ is open in $\dom(X)$.} induced by such a representation $\delta$ is always
the trivial topology on $Y$, i.e., equal to $\{\emptyset,Y\}$. Every multi-valued function $f:\In X\mto Y$ with
a densely realized target space $Y$ is densely realized.
 
\begin{proposition}
\label{prop:indiscrete}
Let $f:\In X\mto Y$ be a multi-valued function on represented spaces where $Y$ is densely realized.
Then $f$ is densely realized.
\end{proposition}
\begin{proof}
Let $(X,\delta_X)$ and $(Y,\delta_Y)$ be represented spaces. 
If $Y$ is densely realized, then $\delta_Y^{-1}(y)$ is dense in $\dom(\delta_Y)$ for all $y\in Y$.
Hence, $\delta_Y^{-1}\circ f\circ\delta_X(p)$ is dense in $\dom(\delta_Y)$ for all
$p\in\dom(f\delta_X)$.
\end{proof}

Densely realized spaces occur quite naturally, for instance all derived spaces have this property.

\begin{example}[Derived space]
If $(X,\delta)$ is a represented space, then the {\em derived space} $(X,\delta')$
with $\delta':=\delta\circ\lim$ is densely realized.
\end{example}

Derived spaces have previously been considered as output spaces for theorems, for instance
the weak Bolzano-Weierstra\ss{} theorem $\WBWT_\IR$ studied in \cite{Kre12} has a 
derived space as its output space.

\begin{example}[Weak Bolzano-Weierstra\ss{} theorem]
The Weak Bolzano-Weierstra\ss{} theorem $\WBWT_\IR$ is densely
realized and hence $\omega$--indiscriminative.
\end{example}

Another interesting case of a densely realized space is the space of Turing degrees.
By $[p]:=\{q\in\IN^\IN:p\equivT q\}$ we denote the Turing degree of $p\in\IN^\IN$.
In a natural way, $p$ can be considered as representative of its degree $[p]$.

\begin{definition}[Turing degrees]
Let $\DD:=\{[p]:p\in\IN^\IN\}$ be the set of Turing degrees with the 
representation $\delta_\DD:\IN^\IN\to\DD,p\mapsto[p]$.
\end{definition}

In the following we understand all computability statements on $\DD$ with respect to the
representation $\delta_\DD$. For every $p\in\IN^\IN$ and every word $w\in\IN^*$
there is a $q\in\IN^\IN$ such that $w\prefix q\in[p]$.
Hence $(\DD,\delta_\DD)$ is a densely realized space and all multi-valued
functions with output $\DD$ are $\omega$--indiscriminative.

\begin{corollary}[Turing degrees]
Let $X$ be a represented space and $f:\In X\mto\DD$ a multi-valued function.
Then $f$ is densely realized and hence $\omega$--indiscriminative.
\end{corollary}

So every for-all-exists theorem that claims the existence of some Turing degree
is automatically $\omega$--indiscriminative.
This means that large parts of computability theory have very little uniform computational
content in terms of the Weihrauch lattice.
As an example we consider the statement that for every ${\mathbf a}\in\DD$ there
exists ${\mathbf b}\in\DD$ such that ${\mathbf b}\nleqT{\mathbf  a}$. 

\begin{example}[Non-computability problem]
The problem 
\[\NON:\DD\mto\DD,{\mathbf a}\mapsto\{{\mathbf b}:{\mathbf b}\nleqT{\mathbf  a}\}\] 
is  $\omega$--indiscriminative.
\end{example}

$\NON$ can be seen as one of the simplest non-computable existence statements asserting the existence of a Turing degree.
It is clear that $\NON$ is not computable, we prove that it is also not continuous.
In general, all $\omega$--discriminative problems are automatically discontinuous (since $\ACC_\IN$ is discontinuous), 
while for $\omega$--indiscriminative problems discontinuity needs to be checked individually. 

\begin{proposition} 
\label{prop:NON-discontinuous}
$\NON$ is discontinuous (even if restricted to an arbitrary cone $\{{\mathbf a}\in\DD:{\mathbf a_0}\leqT {\mathbf a}\}$ with ${\mathbf a_0}\in\DD$).
\end{proposition}
\begin{proof}
Suppose that $F:\IN^\IN\to\IN^\IN$ is a continuous realizer of $\NON$.
Then $F$ is computable with respect to some oracle $q\in\IN^\IN$.
Let $p_0\in\IN^\IN$ be arbitrary and $p:=\langle p_0,q\rangle$. Then $p_0\leqT p$
and $F(p)\leqT p$, in contrast to the assumption that $F$ realizes $\NON$.
\end{proof}

\section{Diagonally Non-Computable Functions}
\label{sec:DNC}

In this section we discuss diagonally non-computable functions. 
By $\varphi:\IN\to\PP$ we denote a standard G\"odel numbering
of the set $\PP$ of partial computable functions $f:\In\IN\to\IN$. 
We recall that a function $q:\IN\to\IN$ is called {\em diagonally non-computable} if
$q(n)\not=\varphi_n(n)$ holds for all $n\in\IN$. Here the inequality $q(n)\not=\varphi_n(n)$
can hold for two reasons: either $\varphi_n(n)$ is not defined or it is defined and has a value
different from $q(n)$. We relativize this problem with respect to some oracle $p$ and some set $X\In\IN$ of values
and we use the relativized G\"odel numbering $\varphi^p$ for this purpose.

\begin{definition}[Diagonally non-computable functions]
Let $X\In\IN$.
We define $\DNC_X:\In\IN^\IN\mto\IN^\IN$ for all $p\in\IN^\IN$ by 
\[\DNC_X(p):=\{q\in X^\IN:(\forall n)\;\varphi^p_n(n)\not=q(n)\}.\]
\end{definition}

It is clear that $\DNC_0$ and $\DNC_1$ are nowhere defined and hence computable. The latter follows since for every $p\in\IN^\IN$ there is
some G\"odel number $n$ such that $\varphi^p_n$ is a total function with constant value $1$.
For every set $X\In\IN$ with at least two values, $\DNC_X$ is total.
It is also clear that the degree of $\DNC_X$ for finite $X$ only depends on the cardinality of $X$
and that $X\In Y$ implies $\DNC_Y\leqSW\DNC_X$.

In the next result we show that $\DNC_X$ is the parallelization of $\ACC_X$.
Via Fact~\ref{fact:ACC-LLPO} an equivalent result has been proved independently by Higuchi and Kihara \cite[Proposition~79]{HK14a}.

\begin{theorem}[Diagonally non-computable functions]
\label{thm:DNC}
$\DNC_X\equivSW\widehat{\ACC_{X}}$ for all $X\In\IN$ with at least two elements.
\end{theorem}
\begin{proof}
We use a representation of $\AA_-(X)$
such that $p$ is a name of a set $A\In X$ if $\range(p)-1=X\setminus A$ (that is, if
$n+1\in\range(p)\iff n\not\in A$).

We first prove $\widehat{\ACC_X}\leqSW\DNC_X$. Given a sequence $(A_i)_{i\in\IN}$
of non-empty sets $A_i\In X$ that miss at most one element, we need to find one element in each $A_i$.
We can assume that each $A_i$ is given by some $p_i\in\IN^\IN$ with $\range(p_i)-1=X\setminus A_i$.
We let $p=\langle p_0,p_1,p_2,...\rangle$. 
Then by the relativized smn-theorem there is a computable function $s:\IN\to\IN$ such that
\[\varphi^p_{s(i)}(n)=\left\{\begin{array}{ll}
x_i & \mbox{if $X\setminus A_i=\{x_i\}$}\\
\uparrow & \mbox{if $X\setminus A_i=\emptyset$}
\end{array}\right..\]
The program with G\"odel number $s(i)$ just has to scan the $i$--th projection $p_i$ of the oracle $p$
in order to find some non-zero element $x_i+1$ and halt with output $x_i$ if there is such an element; otherwise it can search
forever. Then every $q\in\DNC_X(p)$ has the property that $q(s(i))\in A_i$. 
In other words, the computable function $H(q):=q\circ s$ satisfies $H\circ\DNC_X(p)\In\widehat{\ACC_X}(p)$
and hence $\widehat{\ACC_X}\leqSW\DNC_X$.

Now we prove $\DNC_X\leqSW\widehat{\ACC_X}$. Let $p\in\IN^\IN$ be given.
Given $p$ we use a computable function $K$ to compute a name $r=\langle r_0,r_1,r_2,...\rangle=K(p)$ of a sequence
$(A_i)_{i\in\IN}$ of sets $A_i\In X$ with
\[X\setminus A_n:=\{\varphi^p_n(n)\}=\range(r_n)-1\]
for all $n\in\IN$. The function $K$ works such that it starts to compute $\varphi^p_n(n)$ and as
long as no result is available, it writes zeros into $r_n$; as soon as the computation $\varphi^p_n(n)$ halts,
the function $K$ switches to write the value $\varphi^p_n(n)+1$ into $r_n$.
We obtain $\widehat{\ACC_{X}}\circ K(p)=A_0\times A_1\times A_2\times...\In\DNC_X(p)$.
This implies $\DNC_X\leqSW\widehat{\ACC_{X}}$.
\end{proof}

It is clear that $\ACC_2=\C_2$ and it is known that $\widehat{\C_2}\equivSW\WKL$ (see \cite{BG11,BBP12}).
Hence, we obtain the following result.

\begin{corollary}[Diagonally non-computable functions]
\label{cor:WKL-DNC}
$\WKL\equivSW\DNC_2$.
\end{corollary}

We mention that the proof of $\DNC_2\leqSW\C_{2^\IN}\equivSW\WKL$ also follows 
directly from the fact that the set 
\[\DNC_{2}(p)=\{q\in 2^\IN:(\forall n)\;\varphi^p_n(n)\not=q(n)\}\]
is closed and a name for it with respect to negative information can easily
be computed from $p$. 

In order to generalize this observation, we introduce
a generalization of $\WKL$. Let $X\In\IN$. By $\Tr_X$ we denote the
set of trees $T\In X^*$ represented via their characteristic functions
$\chi_T:2^{(X^*)}\to\{0,1\}$.  We call a tree $T\In X^*$ {\em big}, if
it has the following property: if $w$ is a node of an infinite path of $T$,
then all but at most one successor node of $w$ are also on some infinite path of $T$.

\begin{definition}[Weak K\H{o}nig's lemma for big trees]
Let $X\In\IN$. By $\BWKL_X$ we denote the problem 
\[\BWKL_X:\In\Tr_X\mto X^\IN,\]
where $\BWKL_X(T)=[T]$ is the set of infinite paths of $T$
and $\dom(\BWKL_X)$ is the set of big trees $T$ with infinite paths.
\end{definition}

Since all binary trees are automatically big, we have $\BWKL_2=\WKL$.
In general we obtain the following.

\begin{theorem}[Diagonally non-computable functions and big trees]
\label{thm:DNC-WKL}
If $k\geq2$, then $\DNC_k\equivSW\BWKL_k$.
\end{theorem}
\begin{proof}
For each $k$ we can compute
a tree $T\In k^*$ with $\DNC_k(p)=[T]$ from $p$, where 
\[\DNC_k(p)=\{q\in k^\IN:(\forall n)\;\varphi^p_n(n)\not=q(n)\}.\] 
This tree $T$ is 
automatically big and has infinite paths (since $k\geq2$).
Hence $\BWKL_k(T)=\DNC_k(p)$. This yields the reduction $\DNC_k\leqSW\BWKL_k$.

For the other direction it suffices to prove $\BWKL_k\leqSW\widehat{\ACC_k}$
by Theorem~\ref{thm:DNC}. We use some computable standard bijection $w:\IN\to k^*$.
Given a big tree $T\In k^*$ with some infinite path, we need to find such a path.
For each number $n$ we consider the word $w_n=w(n)$ and we assign a set $A_n\In k$ to it: $A_n=k\setminus\{i\}$ for the first $i\in k$ for which $w_nik^\IN\cap[T]=\emptyset$ can be detected
if such an $i$ exists and $A_n=k$ otherwise.
This assignment is computable in the sense that negative information on $A_n$ can be computed from $T$
since the property $w_nik^\IN\cap[T]=\emptyset$ is c.e.\ in $T$ due to compactness of $w_nik^\IN$.
If $w_nk^\IN\cap[T]\not=\emptyset$, then
\[A_n=\{i\in\{0,...,k-1\}:w_nik^\IN\cap[T]\not=\emptyset\}.\]
If $w_nk^\IN\cap[T]=\emptyset$, then $A_n\In k$ contains all but one element of $k$.
Now, given the sequence $(A_n)_n$ one
can determine a sequence $p\in k^\IN$ with $p(n)\in A_n$ for each $n\in\IN$ with the help of $\widehat{\ACC_k}$. 
Starting from the number $n_0$ of the empty word $w$ we can determine a sequence $q\in[T]$ inductively using $p$:
set $q(0):=p(n_0)$; if $q|_i=q(0)...q(i-1)$ has already been determined and $n_i\in\IN$ is such that $w_{n_i}=q|_i$,
then $q(i):=p(n_i)$. Altogether, this yields an infinite path $q\in[T]$ and hence the reduction $\BWKL_k\leqSW\widehat{\ACC_k}$.  
\end{proof}

We use the compactness of the set $\DNC_X(p)$ for finite $X\In\IN$
also for Proposition~\ref{prop:ACC-DNC}. First, we need the following observation.

\begin{lemma}
\label{lem:DNC-intersection}
Let $n\geq1$ and $p_0,...,p_{n-1}\in\IN^\IN$. Then $\bigcap_{i=0}^{n-1}\DNC_{n+1}(p_i)\not=\emptyset$.
\end{lemma}

The proof is immediate.

\begin{proposition}
\label{prop:ACC-DNC}
$\ACC_n\nleqW\DNC_{n+1}$ for all $n\geq2$.
\end{proposition}
\begin{proof}
Let $n\geq 2$ and suppose that $\ACC_n\leqW\DNC_{n+1}$.
Then there are computable functions $H,K:\In\IN^\IN\to\IN^\IN$ such that
$H\langle\id,GK\rangle$ is a realizer of $\ACC_n$ whenever $G$ is a realizer
of $\DNC_{n+1}$. Let $p$ be a name of the set $n=\{0,...,n-1\}$. 
Since the set $\DNC_{n+1}K(p)$ is compact and $H$ is continuous,
it follows that $H$ restricted to $\langle\{p\}\times\DNC_{n+1}K(p)\rangle$
is uniformly continuous. Hence there is $k\in\IN$ such that for every
$q\in\DNC_{n+1}K(p)$ the value of $H\langle p,q\rangle$ is already
determined by prefixes $p|_k$ and $q|_k$ of $p$ and $q$, respectively, i.e., $H\langle p|_k\IN^\IN\times q|_k\{0,...,n\}^\IN\rangle=\{H\langle p,q\rangle\}$.
The values $\varphi_j^{K(p)}(j)$ are defined for $j$ from $D:=\{j\in\IN:j\in\dom(\varphi_j^{K(p)})$ and $j<k\}$.
It follows from the use theorem that a prefix of length $l$ of the oracle $K(p)$ is sufficient
to guarantee that the corresponding computations halt, i.e., that
$j\in\dom(\varphi_j^{K(p)|_l})$ for all $j\in D$. Due to the continuity of $K$
there is a prefix $p|_m$ of $p$ such that $K(p|_m\IN^\IN)\In K(p)|_l\IN^\IN$;
without loss of generality, we can assume $m\geq k$.
There are $p_0,...,p_{n-1}$ with
$p_i|_m=p|_m$ and $\ACC_n(p_i)=\{0,...,n-1\}\setminus\{i\}$ for all $i<n$. 
By Lemma~\ref{lem:DNC-intersection} there is $r\in\bigcap_{i=0}^{n-1}\DNC_{n+1}K(p_i)$ and 
by our choices of $m,l$ and $p_i$, we obtain $K(p_i)|_l=K(p)|_l$ and hence 
$\varphi_j^{K(p_i)}(j)=\varphi_j^{K(p)}(j)$ for all $j\in D$ and $i<n$.
Thus we can choose $q\in\DNC_{n+1}K(p)$ such that $r|_k=q|_k$.
Hence there is some realizer $G$ of $\DNC_{n+1}$ such that $GK(p_i)=r$ for all $i<n$
and $H\langle p_i,GK(p_i)\rangle=H\langle p_i,r\rangle=H\langle p_i,q\rangle$ does not 
depend on $i$, due to our choice of $k$. 
This contradicts that $H\langle p_i,GK(p_i)\rangle\in\ACC_n(p_i)=\{0,...,n-1\}\setminus\{i\}$ for all $i<n$.
\end{proof}

This yields an independent proof of $\ACC_n\nleqW\ACC_{n+1}$ (which is known by Fact~\ref{fact:ACC-chain})
and it also yields the following corollary, which alternatively follows from \cite[Corollary~80]{HK14a} or \cite[Theorem~6]{Joc89}.

\begin{corollary}
\label{cor:DNCN}
$\DNC_\IN\lW\DNC_{n+1}\lW\DNC_n$ for all $n\geq2$.
\end{corollary}

In particular, $\DNC_n$ is indiscriminative for $n\geq 3$.

In \cite[Proposition~9.5]{BBP12} it was proved that $f\nleqT p$ holds for every $f,p\in\IN^\IN$ such that 
$f$ is diagonally non-computable and $p$ is limit computable in the jump. 
Hence we obtain the following result, where $\lim_\J$ denotes the limit operation $\lim$ restricted to
those $p=\langle p_0,p_1,....\rangle$ such that the sequence of Turing jumps $(\J(p_i))_i$ converges.\footnote{The operation $\lim_\J$ was introduced in \cite{BBP12} and is the
limit operation with respect to a topology that was studied under the name $\Pi$--topology 
by Joseph S.\ Miller \cite[Chapter~IV]{Mil02a}. It proved to be useful in the study of $1$--genericity \cite{BHK16}.}

\begin{corollary}
\label{cor:DNC-limJ}
$\DNC_\IN\nleqW\lim_\J$.
\end{corollary}

Since $\C_\IN\leqW\lim_\J$ we also get the following corollary.

\begin{corollary}
\label{cor:DNC-CN}
$\DNC_\IN\nleqW\C_\IN$.
\end{corollary}

\section{Degrees of Complete Extensions of Peano Arithmetic}
\label{sec:PA}

In this section we briefly want to discuss the problem of finding a degree that contains a complete extension of Peano arithmetic
relative to some given input. 
We recall that a Turing degree ${\mathbf a}$ is called a {\em PA--degree relative to} another Turing degree ${\mathbf b}$, 
in symbols ${\mathbf a}\gg{\mathbf b}$, if every ${\mathbf b}$--computable infinite binary tree has an ${\mathbf a}$--computable path.
The degrees ${\mathbf a}\gg{\mathbf 0}$ are exactly the degrees of complete extensions of Peano arithmetic by results of Jockusch, Soare, and Solovay (see for example \cite[Proposition~2]{Joc89}).
We obtain the following well-known characterization of the relation ${\mathbf a}\gg{\mathbf b}$ that was formally introduced by Simpson \cite{Sim77}.

\begin{proposition}
\label{prop:PA}
Let ${\mathbf a}$ and ${\mathbf b}$ be Turing degrees and let $n\geq 1$. Then the following conditions are equivalent:
\begin{enumerate}
\item ${\mathbf a}\gg{\mathbf b}$, every ${\mathbf b}$--computable infinite binary tree has an ${\mathbf a}$--computable path,
\item every ${\mathbf b}$--computable function $f:\In\IN\to\{0,1\}$ has a total ${\mathbf a}$--computable extension,
\item ${\mathbf a}$ is the degree of a function $f:\IN\to\{0,1,...,n\}$ that is diagonally non-computable relative to ${\mathbf b}$. 
\end{enumerate}
\end{proposition}

For the case $n=1$ it is easy to see that $(1)\TO(2)\TO(3)\TO(1)$. The first implication follows since for every ${\mathbf b}$--computable
function $f:\In\IN\to\{0,1\}$ the set $A\In2^\IN$ of total extensions is co-c.e.\ closed in ${\mathbf b}$. The second implication $(2)\TO(3)$ is obvious,
and the latter implication $(3)\TO(1)$ is implicit in the proof of Theorem~\ref{thm:DNC-WKL}.
The equivalence of (3) for $n=2$ and numbers $n>2$ follows from a relativized version of a theorem of Jockusch (and Friedberg) \cite[Theorem~5]{Joc89}.
We now introduce the problem of Peano arithmetic as follows.

\begin{definition}[Peano arithmetic]
We call $\PA:\DD\mto\DD,{\mathbf b}\mapsto\{{\mathbf a}:{\mathbf a}\gg{\mathbf b}\}$
the {\em problem of Peano arithmetic}.
\end{definition}

The following properties of the relation $\gg$ are easy to establish \cite[Theorem~6.2]{Sim77}.

\begin{proposition}
\label{prop:way-below}
Let ${\mathbf a},{\mathbf b},$ and ${\mathbf c}$ be Turing degrees. Then 
\begin{enumerate}
\item ${\mathbf a}\gg{\mathbf b}\TO {\mathbf a}>_{\rm T}{\mathbf b}$,
\item ${\mathbf a'}\gg{\mathbf a}$,
\item ${\mathbf a}\gg{\mathbf b}\geq_{\rm T}{\mathbf c}\TO{\mathbf a}\gg{\mathbf c}$,
\item ${\mathbf c}\geq_{\rm T}{\mathbf a}\gg{\mathbf b}\TO{\mathbf c}\gg{\mathbf b}$.
\end{enumerate}
\end{proposition}

The first condition implies $\NON\leqSW\PA$. It is also clear that $\PA\nleqW\NON$ holds,
since there are non-computable Turing degrees (for instance minimal ones \cite[Corollary~2]{Joc89}) that are not degrees of complete extensions of Peano arithmetic.
Since $\PA$ is densely realized, we also obtain $\ACC_\IN\nleqW\PA$.

If $f:\In X\mto\IN^\IN$ is a multi-valued function, then we denote by 
\[[f]:\In X\mto\DD,x\mapsto \{[p]:p\in f(x)\}\] 
the {\em degree version} of $f$. It is clear that $[f]\leqSW f$ holds in general, and often $f\nleqW[f]$ holds,
since $[f]$ is always densely realized while $f$ might have stronger uniform computational content. 
This is the case for $f=\DNC_n$, and the
characterization from Proposition~\ref{prop:PA}(3) yields the following result.

\begin{corollary}[Peano arithmetic]
\label{cor:PA-DNC}
$\PA\equivSW[\DNC_n]$ for all $n\geq 2$.
\end{corollary}

By $\J_\DD:\DD\to\DD,{\mathbf a}\mapsto{\mathbf a'}$ we denote the Turing jump operator on degrees
(we have $[\J]\equivSW\J_\DD)$. We recall that $p\in\IN^\IN$ is called {\em low}
if $p'\leqT\emptyset'$ holds. The notion is used correspondingly for Turing degrees.
Proposition~\ref{prop:way-below}(2) yields the following
uniform version.

\begin{corollary}
\label{cor:PA-JD}
$\PA\leqSW\J_\DD$ and $\J_\DD\nleqW\PA$.
\end{corollary}

The latter negative result holds since there are low PA-degrees (by an application of the low basis theorem
to the set $\DNC_2(\emptyset)$).
Now we are going to provide an interesting characterization of $\PA$ as an implication.
As a preparation we prove the following result.

\begin{proposition}
\label{prop:PA-CN-WKL}
$\WKL\leqW\C_\IN'*\PA$.
\end{proposition}
\begin{proof}
Given an infinite binary tree $T$ by a name $t\in\IN^\IN$, we can, with the help of $\PA$, obtain a $q\in\IN^\IN$ that is of PA degree relative to $t$.
This $q$ computes an infinite path $p\in2^\IN$ in $T$. 
We use an enumeration of all computable functions $\Phi_n:\In\IN^\IN\to\IN^\IN$.
Then there must be some $n\in\IN$ such that $p=\Phi_n(q)$ is a path in $T$.
We test all numbers $n\in\IN$ in parallel and try to compute longer and longer prefixes of $\Phi_n(q)$.
Whenever a longer prefix of $\Phi_n(q)$ than before lies completely in $T$, we output the number $n$.
Hence, any fixed number $n\in\IN$ will be produced infinitely often if and only if $p=\Phi_n(q)$ is an infinite path in $T$.
It was proved in \cite[Theorem~9.4]{BGM12} that $\C_\IN'$ is equivalent to the cluster point problem on the natural numbers; hence
it can be used to find one number $n\in\IN$ that has been produced infinitely often. Then $p=\Phi_n(q)$ is 
an infinite path in $T$, as desired.  
\end{proof}

As a corollary we obtain the following  characterization of $\PA$.

\begin{theorem}[Peano arithmetic]
\label{thm:PA}
$\PA\equivW(\C_\IN'\to\WKL)$.
\end{theorem}
\begin{proof}
Proposition~\ref{prop:PA-CN-WKL} implies $(\C_\IN'\to\WKL)\leqW\PA$. 
By Corollary~\ref{cor:WKL-DNC} we have $\WKL\equivSW\DNC_2$. 
Now let  $h$ be such that $\DNC_2\leqW\C_\IN'*h$. 
Without loss of generality we can assume that $h$ is of type $h:\In\IN^\IN\mto\IN^\IN$.
Note that $\C_\IN'$ only produces a natural number output.
Now given some $p\in\IN^\IN$, the function $h$ must be able
(potentially after some additional computation) to produce an output $q\in\IN^\IN$
that (potentially after some further computation that uses the discrete output of $\C_\IN'$)  
computes a diagonally non-computable function $f$ relative to $p$. 
Such a function $f$ is of PA degree relative to $p$.
By Proposition~\ref{prop:way-below} (4) we obtain that also $q$ is of PA degree
relative to $p$. Hence $q\in\PA(p)$. This proves $\PA\leqW h$ and hence $\PA\leqW(\C_\IN'*\WKL)$.
\end{proof}

We note that for the direction $\PA\leqW(\C_\IN'*\WKL)$ we have not used any property of $\C_\IN'$ other than that it produces a natural number output. Hence, the same
proof shows $\PA\equivW(\C_\IN^{(n)}\to\WKL)$ for every $n\geq1$.
It is unclear what happens in case of $n=0$.
Using Proposition~\ref{prop:PA-CN-WKL} we can also obtain the following result.

\begin{corollary}
\label{cor:DNC2-k}
$\DNC_2\leqW\C_\IN'*\DNC_k$ for all $k\geq2$.
\end{corollary}

R.\ Friedberg proved that the Turing degrees of $\DNC_2$--functions coincide with the Turing degrees
of $\DNC_k$--functions for all $k\geq2$ \cite[Theorem~5]{Joc89}. Dorais, Hirst and Shafer analyzed 
the uniform content of the equivalence in reverse mathematics under the presence of $\SO{2}$--induction \cite[Theorem~2.7]{DHS15}.
Since $\C_\IN'$ is the counterpart of $\SO{2}$--induction in the Weihrauch lattice, 
Corollary~\ref{cor:DNC2-k} can be seen as a uniform version of their result.
Again it remains unclear whether we can replace $\C_\IN'$ by $\C_\IN$ here.

\section{Martin-L\"of Randomness and Weak Weak K\H{o}nig's Lemma}
\label{sec:MLR-WWKL}

Another problem that is located in the neighborhood of diagonally non-computable functions in the Weihrauch lattice
is Martin-L\"of randomness. By $\MLR:\IN^\IN\mto2^\IN$ we denote the multi-valued function such that 
$\MLR(p)$ contains all $q$ that are Martin-L\"of random relative to $p$ (see \cite{Nie09,DH10} for definitions).
If $p\leqT q$, then $\MLR(q)\In\MLR(p)$.
Since any finite modification of $q\in\MLR(p)$ is also in $\MLR(p)$, we immediately get the following corollary
of Proposition~\ref{prop:dense}.

\begin{lemma}[Martin-L\"of Randomness]
\label{lem:MLR}
$\MLR$ is densely realized and hence $\omega$--indiscriminative.
\end{lemma}

In particular, this implies that $\DNC_\IN\nleqW\MLR$. This is in sharp contrast to the non-uniform situation
where Ku{\v{c}}era \cite{Kuc85} proved that each Martin-L\"of random computes a diagonally non-computable function
(see also \cite[Theorem~8.8.1]{DH10}). Like in Corollary~\ref{cor:DNC-limJ} the result of Ku{\v{c}}era yields the following
corollary with the help of \cite[Proposition~9.5]{BBP12} (see the statement after Corollary~\ref{cor:DNCN}).

\begin{corollary}
\label{cor:MLR-limJ}
$\MLR\nleqW\lim_\J$.
\end{corollary}

The relation between diagonally non-computable functions and Martin-L\"of randomness has also been 
studied in the non-uniform sense of reverse mathematics, for instance by Ambos-Spies et al.\ in \cite{AKLS04}.
We utilize these results in order to show that $\MLR$ and $\DNC_\IN$ are actually incomparable in the 
uniform sense of the Weihrauch lattice.

\begin{proposition}
\label{prop:MLR-DNC}
$\DNC_\IN\nW\MLR$.
\end{proposition}
\begin{proof}
As mentioned above, $\DNC_\IN\nleqW\MLR$ follows from Lemma~\ref{lem:MLR} and Theorem~\ref{thm:DNC}.
By \cite[Theorems~1.4 and 1.8]{AKLS04} there exists a diagonally non-computable $g\in\IN^\IN$ that does not
compute any Martin-L\"of random $r\in2^\IN$.
Hence $\MLR\nleqW\DNC_\IN$.
\end{proof}

Dorais et al.\ \cite{DDH+16} introduced a quantitative version $\varepsilon\dash\WWKL$ of weak weak K\H{o}nig's lemma that was studied further in \cite{BGH15a}.
Here $\varepsilon\dash\WWKL:\In\Tr_2\mto2^\IN,T\mapsto[T]$ is the same problem as $\WKL$, but restricted
to the set $\dom(\varepsilon\dash\WWKL)$ of trees $T$ with measure $\mu([T])>\varepsilon$, where $\mu$ denotes the uniform measure on Cantor space $2^\IN$.
So in particular, $\WWKL:=0\dash\WWKL$.
We also study the problem $(1-*)\dash\WWKL:\In\Tr_2^\IN\mto2^\IN$, introduced in \cite{BGH15a}, which is defined by
\[(1-*)\dash\WWKL((T_n)_n):=\bigsqcup_{n\in\IN}(1-2^{-n})\dash\WWKL(T_n).\]
Intuitively, this problem can be described as follows: given a sequence $(T_n)_n$ of trees with $\mu([T_n])>1-2^{-n}$, we need to
find one path in any one of these trees; that is, we can find $n$ and a path $p\in[T_n]$. Logically this corresponds to a uniform existential quantification
$(\exists n)(1-2^{-n})\dash\WWKL$. It is clear that $(1-*)\dash\WWKL\leqW\varepsilon\dash\WWKL$ for all $\varepsilon<1$.
We obtain the following.

\begin{lemma}
\label{lem:ACCN-WWKL}
$\ACC_\IN\lW(1-*)\dash\WWKL$ and $\ACC_\IN\lSW(1-*)\dash\WWKL$.
\end{lemma}
\begin{proof}
Given a set $A\In\IN$ in which at most one element of $\IN$ is missing and a number $n$,
we consider the subtrees $w_i2^*$ where $\{w_0,...,w_{2^{n+1}-1}\}=\{0,1\}^{n+1}$ is 
the set of all binary words of length $n+1$. We want to construct a sequence of trees $(T_n)_n$ to which we can apply $\lSW(1-*)\dash\WWKL$. If some $i\in\{0,...,2^{n+1}-2\}$ is missing from $A$, 
then we remove the corresponding subtree $w_i2^*$ from the full tree in order to get a tree $T_n$;
if a number $i\geq2^{n+1}-1$ is missing in $A$, then we remove $w_{2^{n+1}-1}2^*$ in order to obtain $T_n$; and if no number $i$ is missing from $A$, then $T_n$ is the full tree $2^*$. 
The map that takes $(n,A)$ (where $A$ is given by negative information) to $T_n$ is computable, and $T_n$ satisfies $\mu([T_n])\geq1-2^{-n-1}>1-2^n$.
Any infinite path $p\in [T_n]$ can be used to identify in a computable way a number $i\in A$. 
Hence $\ACC_\IN\leqSW(1-*)\dash\WWKL$. 
It is clear that the reductions are strict, since $\ACC_\IN$ only produces computable values on computable inputs (in fact, natural numbers),
while there is a computable sequence $(T_n)_n$ of trees $T_n$ with $\mu([T_n])>1-2^{-n}$ and such that no $T_n$ has has an infinite computable path
(such a sequence of trees can be obtained, for instance, by a universal Martin-L\"of test, as explained below).
\end{proof}

In particular, $(1-*)\dash\WWKL$ is $\omega$--discriminative. 
The following result is also easy to obtain.

\begin{lemma}
\label{lem:MLR-WWKL}
$\MLR\lW(1-*)\dash\WWKL$ and $\MLR\lSW(1-*)\dash\WWKL$.
\end{lemma}
\begin{proof}
The reduction can be shown using a universal Martin-L\"of test $(U_n)_n$ in $p$: we can computably convert the $(U_n)_n$ into a sequence $(T_n)_n$ of trees
with $\mu([T_n])>1-2^{-n}$, and any such tree has only infinite paths, which are Martin-L\"of random in $p$ 
\cite[Theorem~12]{BGH15}.
That the reduction is strict follows from Lemmas~\ref{lem:MLR} and \ref{lem:ACCN-WWKL}.
\end{proof}

Next we want to show that also $\DNC_\IN\leqW(1-*)\dash\WWKL$. This follows from Lemma~\ref{lem:ACCN-WWKL}
and the following result.

\begin{proposition}
\label{prop:WWKL-parallel}
$(1-*)\dash\WWKL$ is strongly parallelizable.
\end{proposition}
\begin{proof}
We need to prove $\widehat{(1-*)\dash\WWKL}\leqSW(1-*)\dash\WWKL$.
Given a double sequence $(T_{k,n})$ of binary trees with $\mu([T_{k,n}])>1-2^{-n}$ for all $k,n$ we need to compute
one sequence $(T_n)$ of trees with $\mu([T_n])>1-2^{-n}$ such that from an infinite path $p\in[T_n]$ for an arbitrary $n$,
we can compute one infinite path $p_k\in[T_{k,n_k}]$ for some arbitrary $n_k$ for each $k\in\IN$.
Given the double sequence $(T_{k,n})$ we can compute a sequence $(T_n)$ of trees such that
\[[T_n]=\bigcap_{k=0}^\infty[T_{k,n+k+1}]\]
for all $n$. Then $\mu[T_n]>1-\sum_{k=0}^\infty2^{-n-k-1}=1-2^{-n}$.
Moreover, an infinite path $p\in[T_n]$ for some $n$ is also an infinite path $p\in[T_{k,n+k+1}]$ for all $k$,
which completes the desired reduction.  
\end{proof}

Lemma~\ref{lem:ACCN-WWKL}, Theorem~\ref{thm:DNC} and Proposition~\ref{prop:WWKL-parallel} yield the desired corollary.
That the reduction in the following corollary is strict follows from Lemma~\ref{lem:MLR-WWKL} and Proposition~\ref{prop:MLR-DNC}.

\begin{corollary}
\label{cor:DNC-WWKL}
$\DNC_\IN\lW(1-*)\dash\WWKL$ and $\DNC_\IN\lSW(1-*)\dash\WWKL$.
\end{corollary}

We note that this result shows that $\DNC_\IN$ admits a Las Vegas algorithm in the sense of \cite{BGH15a},
even one of any success probability arbitrarily close to $1$, while $\DNC_\IN$ cannot be reduced to $\MLR$ by Lemma~\ref{lem:MLR}.

We next want to prove that $\PA$ is not reducible to any jump of $\WWKL$.
This can be established using a theorem of Jockusch and Soare \cite[Corollary~5.4]{JS72}.

\begin{lemma}[Jockusch and Soare 1972]
\label{lem:PA-measure}
Let $A=\{B\In\IN:B$ is of PA--degree$\}$. Then $\mu(A)=0$.
\end{lemma}

We recall (see \cite{BP10} and \cite{BGH15a}) that a function $f:\In(X,\delta_X)\mto(Y,\delta_Y)$ on represented spaces is called {\em probabilistic}
if there is a computable function $F:\In\IN^\IN\times2^\IN\to\IN^\IN$ and a family $(A_p)_{p\in D}$ of measurable subsets of $2^\IN$
with $D:=\dom(f\delta_X)$ such that $\mu(A_p)>0$ for all $p\in D$ and $\delta_YF(p,r)\in f\delta_X(p)$ for all $p\in D$ and $r\in A_p$.
Roughly speaking, a function is probabilistic if it can be computed with the help of a piece of random advice originating from
some set of positive measure that can depend non-uniformly and non-effectively on the input. 
We now transfer the proof of \cite[Theorem~20]{BP10} into our setting.

\begin{proposition}
\label{prop:PA-probabilistic}
$\PA$ is not probabilistic.
\end{proposition}
\begin{proof}
Let us assume that $\PA:\DD\mto\DD$ is probabilistic. Then there exists a computable function $F:\In\IN^\IN\times2^\IN\to\IN^\IN$
and a family $(A_p)_{p\in\IN^\IN}$ of measurable sets $A_p\in2^\IN$ such that $\mu(A_p)>0$ for all $p\in\IN^\IN$ and such that $[F(p,r)]\in\PA([p])$ for all $p\in\IN^\IN$ and $r\in A_p$. 
We fix the computable zero sequence $p$. Then we obtain for each $r\in A_p$ that ${\mathbf a}:=[F(p,r)]\in\PA(p)$ is a PA--degree and ${\mathbf a}\leqT [r]$.
Hence $\mu(A_p)=0$ by Lemma~\ref{lem:PA-measure}, which is a contradiction.
\end{proof}

Since probabilistic Weihrauch degrees are closed downwards \cite[Proposition~14.3]{BGH15a} we obtain the following corollary.

\begin{corollary}
\label{cor:DNC-probabilistic}
$\DNC_n$ is not probabilistic for all $n\geq 2$.
\end{corollary}

We note that this contrasts with the situation for $\DNC_\IN$, which is probabilistic by Corollary~\ref{cor:DNC-WWKL} and \cite[Corollary~14.7]{BGH15a}.
Proposition~\ref{prop:PA-probabilistic} also yields the following.

\begin{corollary}
\label{cor:PA-WWKL-jump}
$\PA\nleqW\WWKL^{(k)}$ for all $k\in\IN$. 
\end{corollary}

We note that the proof of \cite[Theorem~10.1]{BGH15a} yields the following result.

\begin{lemma}
\label{lem:ACC-WWKL}
$\ACC_{n+1}\leqSW\frac{n-1}{n}\dash\WWKL$ and $\ACC_n\nleqW\frac{n-1}{n}\dash\WWKL$ for all $n\geq2$.
\end{lemma}

As a side result we obtain the following conclusion from Lemma~\ref{lem:ACC-WWKL}, Corollary~\ref{cor:PA-WWKL-jump}
and Theorem~\ref{thm:DNC}, which contrasts with Proposition~\ref{prop:WWKL-parallel}.

\begin{corollary}
For all $\varepsilon\in[0,1)$, $\varepsilon\dash\WWKL$ is not parallelizable. 
\end{corollary}

\begin{figure}[htb]
\begin{tikzpicture}[scale=.5,auto=left,every node/.style={fill=black!15}]

\def\rvdots{\raisebox{1mm}[\height][\depth]{$\huge\vdots$}};

  \node (J) at (0,17) {$\J$};
  \node (JD) at (-4,16) {$\J_\DD$};
  \node (DNC2) at (0,15) {$\DNC_2\equivSW\WKL$};
  \node (DNC3) at (0,12) {$\DNC_3\equivSW\BWKL_3$};
  \node (DNCn1) at (0,9) {$\DNC_{n+1}\equivSW\BWKL_{n+1}$};
  \node (DNCN) at (3,5) {$\DNC_\IN$};
  \node (PA) at (0,4) {$\PA$};
  \node (LIMJ) at (8,16) {$\lim_\J$};
  \node (CN) at (8,13) {$\C_\IN$};
  \node (ACC2) at (8,11) {$\ACC_2\equivSW\LLPO$};
  \node (ACC3) at (8,9) {$\ACC_3\equivSW\LLPO_3$};
  \node (ACCn1) at (8,7) {$\ACC_{n+1}\equivSW\LLPO_{n+1}$};
  \node (ACCN) at (8,4) {$\ACC_\IN$};
  \node (NON) at (8,2) {$\NON$};
  \node (WWKL) at (16,14) {$\WWKL$};
  \node (12WWKL) at (16,11.5) {$\frac{1}{2}\dash\WWKL$};
  \node (n1WWKL) at (16,9) {$\frac{n-1}{n}\dash\WWKL$};
  \node (1SWWKL) at (16,6) {$(1-*)\dash\WWKL$};
  \node (MLR) at (16,4) {$\MLR$};

  \foreach \from/\to in {
  J/DNC2,
  J/LIMJ,
  J/JD,
  DNC2/DNC3,
  DNC3/DNCn1,
  DNCn1/DNCN,
  DNCN/NON,
  DNC2/WWKL,
  DNC2/ACC2,
  DNC3/ACC3,
  DNCn1/ACCn1,
  DNCn1/PA,
  PA/NON,
  DNCN/ACCN,
  CN/ACC2,
  ACC2/ACC3,
  ACC3/ACCn1,
  ACCn1/ACCN,
  WWKL/12WWKL,
  12WWKL/n1WWKL,
  n1WWKL/1SWWKL,
  1SWWKL/MLR,
  MLR/NON,
  WWKL/ACC2,
  12WWKL/ACC3,
  n1WWKL/ACCn1,
  1SWWKL/DNCN,
  LIMJ/CN}
  \draw [->,thick] (\from) -- (\to);


 \draw [->,thick,looseness=1] (JD) to [out=260,in=180] (PA);


\end{tikzpicture}
  
\caption{Diagonally non-computable functions in the Weihrauch lattice. }
\label{fig:diagram-DNC}
\end{figure}

In order to import further knowledge on diagonally non-computable functions into our lattice, it
is useful to mention the following relation between Weihrauch reducibility and Medvedev reducibility.
We recall that for two sets $A,B\In\IN^\IN$, $A$ is {\em Medvedev reducible to} $B$, written $A\leqM B$,
if there exists a computable function $F:\In\IN^\IN\to\IN^\IN$ such that $F(B)\In A$. For the following result
we need to assume that $g$ has a realizer (or we need to assume the Axiom of Choice for Baire space, which
implies this property).

\begin{lemma}[Weihrauch and Medvedev reducibility]
\label{lem:Medvedev}
For $f,g:\In\IN^\IN\mto\IN^\IN$ (such that $g$ has a realizer) we obtain
\[f\leqW g\TO (\forall\mbox{ computable }p\in\dom(f))(\exists\mbox{ computable }q\in\dom(g))\;f(p)\leqM g(q).\]
\end{lemma}
\begin{proof}
If $f\leqW g$, then there are computable $H,K$ such that $H\langle\id,GK\rangle$ is a realizer of $f$ whenever $G$ is a realizer of $g$.
Let $p\in\dom(f)$ be computable. Then $q:=K(p)$ is computable and $GK(p)\in g(q)$ for every realizer $G$ of $g$.
In fact, for every $r\in g(q)$ there is a realizer $G$ of $g$ such that $GK(p)=r$, and hence $H\langle p,r\rangle\in f(p)$ for all $r\in g(q)$.
In other words, the function $F:\In\IN^\IN\to\IN^\IN$ with $F(r):=H\langle p,r\rangle$ is computable and satisfies $F(g(q))\In f(p)$. 
This means that $f(p)\leqM g(q)$.
\end{proof}

In \cite[Theorem~5.4]{DGJM11} Downey et al.\ proved that the Martin-L\"of random points are not Medvedev reducible to the diagonally non-computable functions with three values.
We note that the Medvedev degrees of $\DNC_3(q)$ are identical for all computable $q$. 
This implies $\MLR\nleqW\DNC_3$ by Lemma~\ref{lem:Medvedev}, and hence together with Proposition~\ref{prop:MLR-DNC} we obtain the following result.

\begin{corollary}
\label{cor:MLR-DNC3}
$\MLR\nW\DNC_3$.
\end{corollary}

In the diagram in Figure~\ref{fig:diagram-DNC} we collect the results on diagonally non-computable functions, weak weak K\H{o}nig's lemma
and Martin-L\"of randomness. All lines indicate strong Weihrauch reductions $\leqSW$ against the direction of the arrow; i.e., if
$f\leqSW g$, then the arrow points from $g$ to $f$.
We note that by Corollary~\ref{cor:PA-WWKL-jump} and Corollary~\ref{cor:MLR-DNC3}
we also get the following separation.

\begin{corollary}
\label{cor:WKL-WWKL}
$\BWKL_n\nW\frac{k-1}{k}\dash\WWKL$ for all $n\geq 3$ and $k\geq2$.
\end{corollary}

\section{The Low Basis Theorem}
\label{sec:LBT}

The purpose of this section is to classify the computational content of 
the low basis theorem of Jockusch and Soare \cite[Theorem~2.1]{JS72}.
It states that every computable infinite binary tree has a low path. We consider the natural relativized version
that states that every computable infinite binary tree has a path that is low relative to the tree.
This version has a straightforward interpretation in the Weihrauch lattice.

\begin{definition}[Low basis theorem]
By $\LBT:\In\Tr_2\mto2^\IN$ we denote the multi-valued function with
\[\LBT(T):=\{p\in[T]:p'\leqT T'\},\]
where $\dom(\LBT):=\{T\in\Tr_2:T$ infinite$\}$.
\end{definition}

The (relativized version of the) classical low basis theorem guarantees
that $\LBT$ is actually well-defined, i.e., $\LBT(T)$ is non-empty whenever
$T$ is an infinite binary tree.
The uniform low basis theorem \cite[Theorem~8.3]{BBP12} can
be used in order to derive a rough classification of $\LBT$.
Here $\Low:=\J^{-1}\circ\lim$ denotes the {\em low map} introduced in \cite{BBP12},
which is the composition of the inverse of the Turing jump operator $\J:\IN^\IN\to\IN^\IN$ with the limit map $\lim$.
Hence $\Low$ takes as input a sequence $\langle p_0,p_1,p_2,...\rangle$ with $p_i\in\IN^\IN$ and maps it
to the $q\in\IN^\IN$ such that $q'=\lim_{i\to\infty}p_i$ (if there is such a $q$, otherwise it is undefined). That is,  
the image of the set of computable points in $\dom(\Low)$ under $\Low$ is exactly the set of low points.

\begin{proposition}
\label{prop:LBT}
$\WKL\leqSW\LBT\leqSW\Low$.
\end{proposition}
\begin{proof}
Firstly, it is clear that $\WKL\leqSW\LBT$ holds, since $\LBT$ is a restriction of $\WKL$ in the image. 
Secondly, the uniform low basis theorem~\cite[Corollary~8.5]{BBP12} states that $\WKL\equivSW\C_{2^\IN}\leqSW\Low$ holds, and hence $\LBT\leqSW\Low$ follows,
since the former shows that $\Low$ computes a realizer of $\WKL$ that produces outputs which are low relative to the input
and since $\LBT$ is the restriction of $\WKL$ in the image to exactly such outputs. 
Hence $\Low$ computes a realizer of $\LBT$.
\end{proof}

In order to separate the low basis theorem $\LBT$ from $\WKL$, we can use the hyperimmune-free basis
theorem of Jockusch and Soare \cite[Theorem~2.4]{JS72}, which states
that every computable infinite binary tree has a path that is of hyperimmune-free degree
(i.e., each function computable from the degree is dominated by a computable function, see Section~\ref{sec:HYP} for precise definitions).
We reformulate this theorem in a way that is directly applicable for our purposes.

\begin{theorem}[Hyperimmune-free basis theorem]
\label{thm:HFBT} 
If $f$ is a multi-valued function on represented spaces with $f\leqW\C_\IR$,
then $f$ has a realizer that maps computable inputs to outputs of hyperimmune-free degree.
\end{theorem}
\begin{proof}
Firstly, $\WKL$ has a realizer which maps computable inputs to outputs of hyperimmune-free degree,
which is a direct consequence of the hyperimmune-free basis theorem of Jockusch and Soare.
Since $\C_\IR\equivSW\C_\IN\times\WKL$ by \cite[Corollary~4.9]{BBP12}, the same holds
for $\C_\IR$, as $np$ is of hyperimmune-free degree if $n\in\IN$ and $p$ is of hyperimmune-free degree.
Since hyperimmune-free Turing degrees are closed downwards with respect to Turing reducibility,
it follows that every $f\leqSW\C_\IR$  has a realizer that maps computable inputs to
output of hyperimmune-free degree. 
Since $\C_\IR$ is a cylinder by \cite[Proposition~8.11]{BBP12}, it follows that $f\leqW\C_\IR$ implies $f\leqSW\C_\IR$,
which proves the result.
\end{proof}

Since there are computable infinite binary trees without computable paths and since 
low points that are non-computable are not of hyperimmune-free degree (see \cite[Proposition~1.5.12]{Nie09}), we can conclude
that $\LBT$ does not have a realizer that maps computable inputs to outputs of hyperimmune-free degree.
This implies that $\LBT\nleqW\C_\IR$ by Theorem~\ref{thm:HFBT}.
Now we also prove a separation in the reverse direction. We recall that $\LPO:\IN^\IN\to\{0,1\}$
is the characteristic function of $\{\widehat{0}\}$.
For any $f:\In X\mto Y$ we denote by $f_\cc$ the restriction of $f$ to computable inputs. 
It is clear that $f\mapsto f_\cc$ is an interior operator in the Weihrauch lattice and, in particular,
$f\leqW g$ implies $f_\cc\leqW g_\cc$.

\begin{proposition}
\label{prop:LPO-LBT}
$\LPO\nleqW\LBT$.
\end{proposition}
\begin{proof}
We claim that $(\widehat{\LBT})_\cc\equivW\LBT_\cc$.
Given a sequence of infinite binary trees $(T_i)_i$ we can compute 
the product tree $T$ with $[T]=\langle [T_0]\times[T_1]\times[T_2]\times...\rangle$.
Then $p=\langle p_0,p_1,...\rangle$ is an infinite path in $T$ if and only if $p_i$ is an infinite path in $T_i$ for
all $i\in\IN$. If the sequence $(T_i)_i$ is computable and $p\in[T]$ is low relative to $T$, then $T$ is computable and
\[\langle p_0',p_1',p_2',...\rangle\leqT \langle p_0,p_1,p_2,...\rangle'=p'\leqT T'\equivT\emptyset',\] 
and hence $p_i'\leqT\emptyset'\equivT T_i'$ for all $i$. This proves $(\widehat{\LBT})_\cc\leqSW\LBT_{\rm c}$.
The inverse reduction is clear.

Let us now assume that $\LPO\leqW\LBT$ holds. It is known that $\widehat{\LPO}\equivW\lim$ (see \cite[Corollary~6.4]{BG11} and \cite[Proposition~9.1]{Bra05}).
We obtain with Proposition~\ref{prop:LBT}
\[\lim\nolimits_\cc\leqW(\widehat{\LPO})_\cc\leqW(\widehat{\LBT})_\cc\leqW\LBT_\cc\leqW\Low_\cc\lW\lim\nolimits_\cc,\]
which is a contradiction. The latter reduction is strict, since there are limit computable $p$ that are not low.
\end{proof}

Since $\LPO\leqW\C_\IN\leqW\C_\IR$, we can conclude that $\C_\IR\nleqW\LBT$ holds.
Since also $\WKL\lW\C_\IR\lW\Low$ is known \cite[Theorem~8.7]{BBP12}, we also obtain $\Low\nleqW\LBT$.
Altogether, the results of this section can be summarized as follows. 

\begin{corollary}[Low basis theorem]
\label{cor:LBT}
$\WKL\lW\LBT\lW\Low$ and $\LBT\nW\C_\IR$.
\end{corollary}

This characterizes the position of $\LBT$ in the Weihrauch lattice relative to its known immediate neighborhood.

\section{The Hyperimmunity Problem}
\label{sec:HYP}

In this section we want to study the hyperimmunity problem.
The statement behind this problem is that for every function
$p:\IN\to\IN$ there exists a function $q:\IN\to\IN$ such that no function $r$ that is computable
in $p$ dominates $q$. Here we say that $r$ {\em dominates} $q$ if it satisfies $(\forall n)\;q(n)\leq r(n)$.
Functions that are not dominated by any computable function 
are called {\em hyperimmune}, and hence the principle can also be stated such that 
for every function $p$ there is a function $q$, which is hyperimmune relative to $p$.

\begin{definition}[Hyperimmunity]
We call $\HYP:\IN^\IN\mto\IN^\IN$ with
\[\HYP(p):=\{q\in\IN^\IN:(\forall r\leqT p)(\exists n)\;r(n)<q(n)\}\]
the {\em hyperimmunity problem}.
\end{definition}

It is clear that $\HYP(p)$ contains exactly all hyperimmune $q$ if $p$ is computable.
We will implicitly prove below that $\HYP$ is actually total.\footnote{A referee noted that one can also see this directly,
since there are only countably many $r\leqT p$, and a diagonalization argument shows that there is a $q$ that dominates all these $r$'s eventually.}

Now we want to compare the hyperimmunity problem with the weak 1--genericity problem, since it is known
by a result of Kurtz that hyperimmune and weakly 1--generic degrees coincide \cite[Theorem~2.24.14]{DH10}.
We recall that $q\in2^\IN$ is called {\em weakly $1$--generic} in $p\in2^\IN$
if $q\in U$ for each dense set $U\In2^\IN$ that is c.e.\ open in $p$ (see \cite[Definition~1.8.47]{Nie09}). 

\begin{definition}[Weak 1-genericity]
By $1\dash\WGEN:2^\IN\mto2^\IN$ we denote the problem
\[1\dash\WGEN(p):=\{q:\mbox{$q$ is weakly $1$--generic in $p$}\}.\]
\end{definition}

It is a well-known result, due to Kurtz, that every weakly $1$--generic point $q\in2^\IN$ is hyperimmune \cite{Kur81} (see also \cite[Proposition~1.8.49]{Nie09} and \cite[Theorem~2.24.12]{DH10}).
We follow this idea and show that it also holds uniformly. 

\begin{proposition}
\label{prop:HYP-1WGEN}
$\HYP\leqSW1\dash\WGEN$.
\end{proposition}
\begin{proof}
Given some $p\in\IN^\IN$ we use $1\dash\WGEN$ and some computable standard embedding $\iota:\IN^\IN\into 2^\IN$ 
to find some point $s\in1\dash\WGEN(\iota(p))$ that is weakly $1$--generic in $p$.
We consider $s\in2^\IN$ as the characteristic function of the set $A:=s^{-1}\{1\}$ and we compute $q:=p_A$, 
the principal function of $A$ (which is the strictly monotone function that enumerates the elements of $A$).
Let $r:\IN\to\IN$ be an arbitrary function that is computable from $p$. For every word $w\in\{0,1\}^*$ we define $n_{r,w}:=r(|w|)+1$
and
\[U_r:=\bigcup_{w\in\{0,1\}^+}w0^{n_{r,w}}\IN^\IN.\]
Then $U_r\In2^\IN$ is a dense set that is c.e.\ open in $r$ and hence in $p$. Since $s$ is weakly $1$--generic in $p$,
it follows that $s\in U_r$, and hence there is some $w\in\{0,1\}^+$ such that $w0^{n_{r,w}}\prefix s$.
Let $n:=|w|$. Then $q(n-1)=p_A(n-1)\geq n-1$ and hence $q(n)=p_A(n)\geq n_{r,w}=r(n)+1>r(n)$.
Hence $q$ is hyperimmune relative to $p$, i.e., $q\in\HYP(p)$.
\end{proof}

We note that the proof implicitly includes a proof that $\HYP$ is actually total.
We prove that we get at least an ordinary Weihrauch reduction in the opposite direction.
The result that hyperimmune and weakly $1$--generic degrees coincide is due to Kurtz, and 
the following proof is essentially a uniformized (and simplified) version of the proof of 
\cite[Theorem~2.24.14]{DH10}.

\begin{proposition}
\label{prop:1WGEN-HYP}
$1\dash\WGEN\leqW\HYP$.
\end{proposition}
\begin{proof}
Let $p\in2^\IN$. Then with the help of $\HYP$ we can find a $q\in\IN^\IN$ that is hyperimmune
relative to $p$, i.e., no function computable from $p$ dominates $q$.
We now describe a computable function $H$ that can compute from $p$ and any such $q$ an $r\in2^\IN$ that is weakly $1$--generic 
relative to $p$. With this function we obtain that $H\langle \id,G\rangle$ is a realizer of $1\dash\WGEN$ whenever $G$ is a realizer of $\HYP$.

Without loss of generality, we can assume that $q$ is increasing. 
From $p$ we can compute a sequence $(f_i)_i$ of functions  $f_i:\IN\to\{0,1\}^*$ such that $(S_i)_i$ with $S_i:=\range(f_i)$ is an enumeration of all
subsets of $\{0,1\}^*$ that are c.e.\ relative to $p$. We let 
\[S_i[s]:=\{\sigma\in\{0,1\}^*:(\exists j\leq q(s))\;f_i(j)=\sigma\}.\]
Now we compute a sequence $\sigma_s\in\{0,1\}^*$ of words that will converge to $r$.
In order to ensure that $r$ is weakly $1$--generic relative to $p$, it is sufficient to satisfy the requirements
\[R_i:\mbox{ If $S_i2^\IN$ is dense in $2^\IN$, then $(\exists \sigma\in S_i)\;\sigma\prefix r$.}\]
We say that $R_i$ requires attention at stage $s$ if $(\forall\sigma\in S_i[s])\;\sigma\not\prefix\sigma_s$ and $(\exists\sigma\in S_i[s])\;\sigma_s\prefix\sigma$.
With the help of $p$ and $q$, we can decide whether a requirement $R_i$ requires attention at a certain stage.

Now we describe the algorithm that computes $r$ in stages $s=0,1,2,...$. 
At Stage $0$ we set $\sigma_0:=q(0)$. 
At Stage $s+1$, if no requirement $R_i$ with $i\leq s+1$ requires attention, then we let $\sigma_{s+1}:=\sigma_s$.
Otherwise, let $R_i$ be the strongest requirement that requires attention and let $m$ be the least value such that 
$\sigma_s\prefix f_i(m)$. If $|f_i(m)|>s+1$, then we let $\sigma_{s+1}:=\sigma_s$ and otherwise we let 
$\sigma_{s+1}:=f_i(m)$.

We need to prove that $r$ is weakly $1$--generic. Assume that $S_i$ is dense. Let $g(s)$ be the least $k$ such that 
for each $\sigma\in\{0,1\}^{s+1}$ there is a $j\leq k$ with $\sigma\prefix f_i(j)$. Then $g\leqT p$.
Let $s$ be a stage after which no requirement stronger than $R_i$ ever requires attention.
Since $q$ is not dominated by any function computable from $p$, there is a $t>s$ such that $q(t)>g(t)$.
At Stage $t$ either $R_i$ is already satisfied or it requires attention and continues to require attention until it is met.
Thus for some $t^\prime\geqslant t$ there exists $\sigma\in S_i[t^\prime]$, and hence in $S_i$, such that $\sigma\prefix r$.
\end{proof}

We obtain the following corollary.

\begin{corollary}
\label{cor:HYP-1WGEN}
$\HYP\equivW1\dash\WGEN$.
\end{corollary}

This immediately raises the following question.

\begin{question}
$\HYP\equivSW1\dash\WGEN$?
\end{question}

By the hyperimmune-free basis theorem~\ref{thm:HFBT} it is clear that we obtain
the following corollary.

\begin{corollary}
\label{cor:HYP-CR}
$\HYP\nleqW\C_\IR$.
\end{corollary}

In the following we also use the problem of 1--genericity.
For each $p\in2^\IN$ we consider some fixed enumeration $(U_i^p)_{i\in\IN}$ of all sets $U_i^p\In2^\IN$ that are c.e.\ open in $p$.
A point $q\in2^\IN$ is called {\em 1--generic}  in $p\in2^\IN$, if for all $i\in\IN$ there exists some $w\prefix q$ such that
$w2^\IN\In U_i^p$ or $w2^\IN\cap U_i^p=\emptyset$.
It follows directly from this definition that every point $q\in2^\IN$ which is $1$--generic in $p$ is also
weakly $1$--generic in $p$.
We call $q$ {\em 1--generic} if it is $1$--generic in some computable $p\in2^\IN$.
We use the concept of $1$--genericity in order to define the problem $1\dash\GEN$ of {\em 1--genericity}.

\begin{definition}[Genericity]
We define $1\dash\GEN:2^\IN\mto2^\IN$ by 
\[1\dash\GEN(p):=\{q:q\mbox{ is $1$--generic in $p$}\}\]
for all $p\in2^\IN$.
\end{definition}

Since $\MLR\leqW\WWKL\leqW\C_\IR$ by Lemma~\ref{lem:MLR-WWKL} and $\HYP\leqW1\dash\GEN$ by Proposition~\ref{prop:HYP-1WGEN}
and Corollary~\ref{cor:HYP-CR}, 
we obtain that $1\dash\GEN\nleqW\MLR$.
By Corollary~\ref{cor:MLR-limJ} we have $\MLR\nleqW\lim_\J$ and since $1\dash\GEN\leqW\lim_\J$ by \cite[Corollary~9.7]{BHK16}, 
we also obtain that $\MLR\nleqW1\dash\GEN$. Altogether, we have the following corollary.

\begin{corollary}
\label{cor:MLR-1GEN}
$\MLR\nW1\dash\GEN$.
\end{corollary}

\section{The Kleene-Post Theorem}
\label{sec:KPT}

A basic theorem in computability theory is the Kleene-Post theorem, which shows
that Turing reducibility does not generate a linear order, i.e., there are Turing
incomparable degrees (see \cite{KP54} or \cite[Theorem~V.2.2]{Odi89}).

\begin{theorem}[Kleene and Post 1954]
There exist $p,q\in\IN^\IN$ such that $p\nT q$.
\end{theorem}

We consider this theorem as a computational problem in the following more general sense:
for any given $r$ we want to find two incomparable 
degrees above it (alternatively, one could also impose an upper bound here, for instance $r'$, but that
would yield a different problem).

\begin{definition}[Theorem of Kleene and Post]
By $\KPT:\IN^\IN\mto\IN^\IN$ we denote the function with
\[\KPT(r):=\{\langle p,q\rangle\in\IN^\IN:r\leqT p,r\leqT q\mbox{ and }p\nT q\}\]
for all $p\in\IN^\IN$.
\end{definition}

In order to prove that $\KPT$ is reducible to $\MLR$ we use the following version of 
van Lambalgen's theorem \cite{vLam90}.

\begin{theorem}[van Lambalgen's theorem 1990]
\label{thm:van-Lambalgen}
Let $p,q,r\in2^\IN$. If $\langle p,q\rangle$ is Martin-L\"of random in $r$, then
$q$ is Martin-L\"of random in $r$ and $p$ is Martin-L\"of random in $\langle q,r\rangle$.
\end{theorem}

The proofs given in \cite[Theorem~6.9.1]{DH10} or \cite[Theorem~3.4.6]{Nie09} relativize in the stated sense.
It is well-known in computability theory that any Martin-L\"of random yields 
an incomparable pair of degrees due to van Lambalgen's theorem (see \cite[Corollary~6.9.4]{DH10}).
We just translate this observation into our setting.\footnote{A strong Weihrauch reduction does not hold here as pointed out by a referee, 
since only measure zero many sets compute a non-computable $r$, hence most Martin-L\"of randoms relative to $r$
cannot produce anything above $r$.} 

\begin{proposition}
\label{prop:KPT-MLR}
$\KPT\leqW\MLR$.
\end{proposition}
\begin{proof}
Given $r$ we obtain $\langle p,q\rangle\in\MLR(r)$, and by the theorem of van Lambalgen we obtain that $q$ is random relative to $\langle p,r\rangle$
and $p$ is random relative to $\langle q,r\rangle$. Hence $q\nleqT\langle p,r\rangle$ and $p\nleqT\langle q,r\rangle$, which implies that $\langle p,r\rangle$
and $\langle q,r\rangle$ are incomparable, and both are clearly above $r$.
\end{proof} 

Likewise, we can use the analogue of van Lambalgen's theorem for genericity to prove a similar statement for 1--genericity.
We first formulate a suitable relativized version of a theorem of Yu (see \cite{Yu06a} and \cite[Theorem~8.20.1]{DH10}).

\begin{theorem}[Yu's theorem 2006]
\label{thm:Yu}
Let $p,q,r\in2^\IN$. If $\langle p,q\rangle$ is 1--generic in $r$, then
$q$ is 1--generic in $r$ and $p$ is 1--generic in $\langle q,r\rangle$.
\end{theorem}

The proof is a direct relativization of the proof given in \cite[Theorem~8.20.1]{DH10}.
We obtain the following corollary.

\begin{corollary}
\label{cor:KPT-GEN}
$\KPT\leqW1\dash\GEN$.
\end{corollary}

Since $\MLR$ and $1\dash\GEN$ are incomparable, by Corollary~\ref{cor:MLR-1GEN} 
both reductions in Proposition~\ref{prop:KPT-MLR} and Corollary~\ref{cor:KPT-GEN} are strict.
We note that the theorem of Yu~\ref{thm:Yu} has another interesting consequence: it implies that $1\dash\GEN$ is closed under composition.\footnote{Likewise it was
observed by Brattka, Gherardi and H\"olzl (unpublished notes) that van Lambalgen's theorem implies the closure of Martin-L\"of randomness under
composition; that is, $\MLR*\MLR\equivW\MLR$.}

\begin{proposition}
\label{prop:1GEN-composition}
$1\dash\GEN*1\dash\GEN\equivW1\dash\GEN$.
\end{proposition}
\begin{proof}
It is clear that $1\dash\GEN\leqW1\dash\GEN*1\dash\GEN$.
Let $f\leqW1\dash\GEN*1\dash\GEN$. Without loss of generality, we can assume that $f$
is of type $f:\In2^\IN\mto2^\IN$. Hence there are computable single-valued functions $F,G,H:\In2^\IN\to2^\IN$
such that any $q\in1\dash\GEN(G(r))$ has the property that any $p\in1\dash\GEN(H\langle q,r\rangle)$ yields some $F\langle p,q,r\rangle\in f(r)$.
Hence, if $\langle p,q\rangle\in1\dash\GEN(r)$, then by Yu's theorem~\ref{thm:Yu} it follows that
$q\in1\dash\GEN(r)$ and $p\in1\dash\GEN(\langle q,r\rangle)$. Now, $G(r)\leqT r$ and $H\langle q,r\rangle\leqT\langle q,r\rangle$,
and hence $1\dash\GEN(r)\In1\dash\GEN(G(r))$ and $1\dash\GEN(\langle q,r\rangle)\In1\dash\GEN(H\langle q,r\rangle)$.
This implies $F\langle p,q,r\rangle\in f(r)$. Thus the function $F$ yields a reduction $f\leqW1\dash\GEN$. 
\end{proof}

\section{The Jump Inversion Theorem}
\label{sec:JIT}

In this section we study the uniform computational content of Friedberg's jump inversion theorem, which 
is an interesting example since it is continuous but not computable. 
Friedberg's jump inversion theorem
in its original formulation reads as follows (see \cite{Fri57}).

\begin{theorem}[Friedberg's jump inversion theorem]
\label{thm:JIT}
For every degree ${\mathbf a}\in\DD$ there exists a degree ${\mathbf b}\in\DD$ with ${\mathbf b'}={\mathbf a}\cup0'$.
\end{theorem}

In particular, this theorem implies that the Turing jump operator on degrees $\J_\DD:\DD\to\DD,{\mathbf a}\mapsto {\mathbf a'}$ 
is surjective onto the upper cone $\{{\mathbf a}\in\DD:0'\leqT {\mathbf a}\}$.
We can formalize the jump inversion theorem as follows. 

\begin{definition}[Jump inversion theorem]
We call 
\[\JIT:\DD\mto\DD,{\mathbf a}\mapsto\{{\mathbf b}\in\DD:{\mathbf b'}={\mathbf a}\cup0'\}\]
the {\em jump inversion theorem}.
\end{definition}

The inverse $\J_\DD^{-1}$ of the Turing jump operator is then a restriction of $\JIT$
to the upper cone $\{{\mathbf a}\in\DD:0'\leqT {\mathbf a}\}$, and hence it is clear that $\J_\DD^{-1}\lSW\JIT$ holds
(the reduction is strict, since $\J_\DD^{-1}$ has no computable points in its domain).
It follows directly from Proposition~\ref{prop:indiscrete} that $\JIT$, $\J_\DD$ and $\J_\DD^{-1}$
are $\omega$--indiscriminative.  
We will see that $\JIT$ is also indiscriminative for a different reason, namely
it is even continuous. This follows from the following result that is a direct consequence
of the classical proof of the Friedberg jump inversion theorem.

As usual we denote by $\varphi^p_i(n)$ the $i$--th partial computable function relative to $p\in\IN^\IN$ on input $n\in\IN$.
We also use the notation $\varphi^\sigma_i(n)$ for a partial oracle $\sigma\in\IN^*$.
The value  $\varphi^\sigma_i(n)$, for $\sigma\in\IN^*$, is undefined if the computation consults the oracle beyond  
bit $\vert\sigma\vert$ and equals the output otherwise. For simplicity, we assume that the Turing degrees are represented by $2^\IN$
in this section. 
This does not make any essential difference, since there is a computable embedding $\iota:\IN^\IN\into 2^\IN$.
As usual, for two words $\sigma,\tau\in\{0,1\}^*$ the concatenation is denoted by $\sigma^\frown\tau$, and similarly $\sigma^\frown p$ denotes the concatenation of $\sigma$
with $p\in2^\IN$.
We denote by $\cc_{\emptyset'}:\IN^\IN\to\IN^\IN$ the 
constant function with the value of the halting problem.

\begin{proposition}
\label{prop:JIT-cKid}
$\JIT\lSW \cc_{\emptyset'}\times\id$.
\end{proposition}
\begin{proof}
We briefly recall the proof of Theorem~\ref{thm:JIT} in \cite[Theorem~V.2.24]{Odi89}.
Given $A\In\IN$ we construct the characteristic function $\chi_B$ of $B\In\IN$ using the finite
extension method. That is, we define a monotone sequence $(\sigma_n)_n$ of words $\sigma_s\in\{0,1\}^*$ with $\chi_B=\sup_n\sigma_n$ inductively in stages $s=0,1,...$ as follows. 
We let $\sigma_0$ be the empty sequence.
If $\sigma_s$ is already given, then we continue as follows.
\begin{itemize}
\item If $s=2i$ and there is $\sigma\in\{0,1\}^*$ with $\sigma_s\prefix \sigma$ such that
        $\varphi_i^\sigma(i)$ is defined, then we let $\sigma_{s+1}:=\sigma$ for the smallest
        such $\sigma$ and otherwise $\sigma_{s+1}:=\sigma_s$.
\item If $s=2i+1$ then we define $\sigma_{s+1}:=\sigma_s^\frown\chi_A(i)$.
\end{itemize}
The first condition is computable in $\emptyset'$ and the second one in $A$.
Therefore $(\sigma_n)_n$, and hence of $B$, is computable in $\emptyset'\oplus A$.
This shows that the function $F:2^\IN\to2^\IN$ that maps $A$ to $B$ is strongly reducible to $\cc_{\emptyset'}\times\id$.
We still need to show that $F$ realizes $\JIT$.
The first condition guarantees $i\in B'\iff\varphi_i^{\sigma_{2i+1}}(i)\downarrow$
and hence $B'\leqT \emptyset'\oplus A$. On the other hand, 
$i\in A\iff\sigma_{2i+1}(|\sigma_{2i+1}|)=1$ and hence $A\leqT B'$. Consequently $\emptyset'\oplus A\leqT B'$, which completes the proof
of $\JIT\leqSW\cc_{\emptyset'}\times\id$.
The reduction is strict by Proposition~\ref{prop:dense} since $\JIT$ is densely realized
(alternatively it is strict because $\cc_{\emptyset'}\times\id$ does not map computable inputs to computable outputs, but $\JIT$ has computable outputs for any computable input).
\end{proof}

We recall that $f$ is called {\em limit computable} if $f\leqW\lim$ holds.
The proof of Proposition~\ref{prop:JIT-cKid} shows not only that $\JIT$ is limit computable, but also that it has a limit computable realizer
whose range only contains $1$--generic points. This is because the conditions used to construct $B$ in the proof ensure that $B$ is $1$--generic.
Hence we also obtain the following corollary by \cite[Proposition~9.5]{BHK16}, which states that every $f$ that has some limit computable realizer
whose range only contains 1--generic points satisfies $f\leqSW\lim_\J$.

\begin{corollary}
\label{cor:JIT-limJ}
$\JIT\lSW\lim_\J$.
\end{corollary}

The reduction is strict by Proposition~\ref{prop:dense} since $\JIT$ is indiscriminative.
Another conclusion that can be drawn from Proposition~\ref{prop:JIT-cKid} is that $\JIT$ is continuous, because
$\cc_{\emptyset'}\times\id$ is continuous, and this property is preserved downwards by $\leqW$.
An obvious question is now whether $\JIT$ is perhaps even computable?
Given $a=a_0...a_n\in\IN^*$ we denote by $2a:=a_0a_0...a_na_n$ the 
word where each symbol of $a$ is doubled. Analogously, we define $2\alpha$ for $\alpha\in\IN^\IN$.

\begin{proposition}
\label{prop:JD1-non-computable}
$\J_\DD^{-1}$ is not computable.
\end{proposition}
\begin{proof}
For this proof we represent $\DD$ by $2^\IN$ in some effective way.
Suppose there is a computable Turing functional $\varphi=\varphi_i:\In2^\IN\to2^\IN$ that realizes $\J_\DD^{-1}$.
Then $\alpha\nequivT\beta$ implies $\varphi^\alpha\not=\varphi^\beta$ for all $\alpha,\beta\in2^\IN$
that compute the halting problem,
since $(\varphi^\alpha)'\equivT\alpha$ and $(\varphi^\beta)'\equivT\beta$.
Hence, for all words $\sigma\in\{0,1\}^*$ there exist $p,q\in2^\IN$ and $n\in\IN$
such that $\varphi^{\sigma^\frown 2p}(n)\not=\varphi^{\sigma^\frown 2q}(n)$ and such that both values exist.
We use some fixed effective enumeration of $\{0,1\}^*\times\{0,1\}^*\times\IN$.
We inductively construct a computable perfect splitting tree $f:\{0,1\}^*\to\{0,1\}^*$ for $\varphi$
with branches of arbitrarily high Turing degree. To begin with, $f$ takes the empty word to the
empty word. Now suppose we have defined $f$ on all words of length $n$ and
let $\sigma\in\{0,1\}^n$. We define
\[f(\sigma^\frown d):=\left\{\begin{array}{ll}
  f(\sigma)^\frown 2a^\frown 01 & \mbox{if $d=0$}\\
  f(\sigma)^\frown 2b^\frown 10 & \mbox{if $d=1$}
\end{array},\right.\]
for $d\in\{0,1\}$, where $(a,b,s)\in\{0,1\}^*\times\{0,1\}^*\times\IN$ is minimal in our fixed enumeration
such that there exists some $n$ with $\varphi^{f(\sigma)^\frown 2a}(n)[s]\not=\varphi^{f(\sigma)^\frown 2b}(n)[s]$ and such that both values exist.
This completes the construction of $f$. The prefix closure $T\In\{0,1\}^*$ of the image
$f(\{0,1\}^*)$ is a computable binary tree such that for all incompatible $a,b\in T$
there exists some $n\in\IN$ with $\varphi^a(n)\not=\varphi^b(n)$ (and such that both values exist).
Moreover, we can injectively map the Turing semi-lattice into $[T]$ (the set of infinite paths of $T$)
via the pairs $01,10$ that have been added in the construction of $f$. 

Now let $\alpha\in[T]$ be such that $\emptyset'\leqT\alpha$. 
This guarantees that $\alpha\in\dom(\varphi)$.
We show that $\alpha$
is computable from $\varphi^\alpha$, which contradicts the assumption that $\varphi$
realizes $\J_\DD^{-1}$ (i.e., the fact that $(\varphi^\alpha)'\equivT\alpha$).
We compute a monotone increasing sequence $(\sigma_n)_n$ of words $\sigma_n\in\{0,1\}^n$
such that $\alpha=\sup_n f(\sigma_n)$. Let $\sigma_0$ is the empty word. 
Since $\alpha\in[T]$, one of $f(0)$ and $f(1)$ is compatible with $\alpha$.
By construction, we can find some $n\in\IN$ such that $\varphi^{f(0)}(n)\not=\varphi^{f(1)}(n)$
(and such that both values exist), and so by comparing $\varphi^\alpha(n)$ with $\varphi^{f(0)}(n)$
and $\varphi^{f(1)}(n)$ we can decide which of $f(0)$ and $f(1)$ is a prefix of $\alpha$.
We let $\sigma_1:=i$ for the corresponding $i$ with $\varphi^{\alpha}(n)=\varphi^{f(i)}(n)$.
An analogous construction works on all levels: having constructed $\sigma_n$, we have that
$f(\sigma_n^\frown i)$ is an initial segment of $\alpha$ for the $i$ such that
$\varphi^\alpha(n)=\varphi^{f(\sigma_n^\frown i)}(n)\not=\varphi^{f(\sigma_n^\frown (1-i))}(n)$ for some $n$ (where all
these values exist). Since $\alpha=\sup_nf(\sigma_n)$, we obtain 
$\alpha\leqT\varphi^\alpha$.
 \end{proof}

As a direct consequence of Propositions~\ref{prop:JIT-cKid} and \ref{prop:JD1-non-computable}
we obtain the following corollary.

\begin{corollary}
\label{cor:JIT} 
$\JIT$ is limit computable and continuous, but not computable. 
\end{corollary}

We note that $\J_\DD^{-1}\circ\J_\DD^{-1}=0$ (where $0$ denotes the nowhere defined function),
since by a version of the jump inversion theorem due to Cooper (see \cite{Coo73} and also
\cite[Theorem~2.18.7]{DH10}) for every $a\geq 0'$ there is some minimal $b$ with $b'=a$
and hence $b\ngeq 0'$. In fact, we can even formulate the following stronger observation
as a corollary of Cooper's theorem.

\begin{corollary}
$\J_\DD^{-1}\stars\J_\DD^{-1}\equivSW0$.
\end{corollary}

Arno Pauly (personal communication) noted that the following corollary follows from Proposition~\ref{prop:JIT-cKid} 
(using the additional observation that $\cc_{\emptyset'}*\cc_{\emptyset'}\equivW\cc_{\emptyset'}$).
This answers an open question from an earlier version of this article.

\begin{corollary}
$\JIT*\JIT\leqW \cc_{\emptyset'}\leqW\lim$.
\end{corollary}

\section{The Cohesiveness Problem}
\label{sec:COH}

In this section we want to discuss some properties of the cohesiveness problem.
A set $X\In\IN$ is called {\em cohesive} for a sequence $(R_i)_i$ of sets $R_i\In\IN$ if it is infinite
and for each $i\in\IN$ we have $X\In^*R_i$ or $X\In^*(R_i)^{\rm c}$.
Here we write $X\In^*Y$ if $X\setminus Y$ is finite, i.e., if $X$ is included in $Y$ with only finitely many exceptions.
One can prove that for every sequence $(R_i)_i$ of sets there is always a cohesive set $X$.

\begin{definition}[Cohesiveness Problem]
By $\COH:(2^\IN)^\IN\mto2^\IN$ with
\[\COH(R_i)_i:=\{X\In\IN:X\mbox{ cohesive for }(R_i)_i\}\]
we denote the {\em cohesiveness problem}.
\end{definition}

The cohesiveness problem has been introduced into reverse mathematics by \cite{CJS01}, and the Weihrauch degree
of $\COH$ has already been studied in \cite{DDH+16}. In computability theory a set is called {\em r-cohesive} if it is 
cohesive for the sequence of all computable sets, {\em p-cohesive} if it is cohesive for the sequence of all primitive recursive sets,
and {\em cohesive} if it is cohesive for the sequence of all c.e.\ sets \cite{JS93}.
We say that a Turing degree has property $P$ if it has a member with property $P$, and we extend
the different notions of cohesiveness to degrees in this way.
By \cite[Corollary~2.4]{JS93} the r-cohesive Turing degrees coincide with the cohesive ones.
In order to capture the notion of cohesiveness, we introduce the following variant of $\COH$.

\begin{definition}
By $\COH_+$ we denote the map $\COH_+:\AA_+(\IN)^\IN\mto2^\IN$ with $\COH_+(R_i)_i:=\COH(R_i)_i$.
\end{definition}

We note that the objects in the domain of $\COH$ and $\COH_+$ are the same, they are just represented in different ways.
While $2^\IN$ can be identified with $\AA(\IN)$, the set of closed subsets $A\In\IN$ represented by full information (i.e., by characteristic functions),
the set $\AA_+(\IN)$ captures the set of subsets $A\In\IN$ represented by positive information (i.e., enumerations). 
We will see that $\COH$ can be seen as the uniform version of p--cohesiveness, whereas
$\COH_+$ captures cohesiveness (or r--cohesiveness).
It is also easy to see that $\COH_+$ is located between $\COH$ and $\COH'$.

\begin{proposition}
\label{prop:COH-COHD}
$\COH\leqSW\COH_+\leqSW\COH'$.
\end{proposition}
\begin{proof}
Since $\id:\AA(\IN)\to\AA_+(\IN)$ is computable and $\id:\AA_+(\IN)\to\AA(\IN)$ is limit computable by \cite[Proposition~4.2]{BG09}
we obtain $\COH\leqSW\COH_+\leqSW\COH'$.
\end{proof}

Next we want to apply a very useful characterization of cohesive degrees of Jockusch and Stephan \cite{JS93}\footnote{We note that none of the results from \cite{JS93} that we use here are affected by the correction \cite{JS97}.}
in order to separate $\COH$ from $\COH_+$.
The following is a corollary of the proof of \cite[Theorem~2.1]{JS93}.

\begin{corollary}[Jockusch and Stephan]
\label{cor:Jockusch-Stephan}
Let ${\mathbf a,\mathbf b}$ be Turing degrees. Then the following are equivalent:
\begin{enumerate}
\item ${\mathbf a}$ is cohesive for every ${\mathbf b}$--computable sequence $(A_i)_i$ of sets $A_i\In 2^\IN$,
\item ${\mathbf a'}\gg{\mathbf b'}$.
\end{enumerate}
\end{corollary}

By Corollary~\ref{cor:Jockusch-Stephan} and by using ideas from the proof of \cite[Theorem~2.9(ii)]{JS93} we obtain the 
following separation result.

\begin{proposition}
\label{prop:COH-COHD-separation}
$\COH_+\nleqW\COH$.
\end{proposition}
\begin{proof}
Suppose that $\COH_+\leqW\COH$. Then there are computable functions $H,K$ such that
$H\langle\id,GK\rangle$ is a realizer of $\COH_+$ whenever $G$ is a realizer of $\COH$.
Let $p\in\IN^\IN$ be a computable name for the sequence of all c.e.\ subsets $B\In\IN$,
and consider the sequence $(A_i)_i$ of computable sets given by $K(p)$. By the proof of \cite[Theorem~2.9(ii)]{JS93} there is a Turing degree ${\mathbf a}$ 
that is not cohesive, but such that ${\mathbf a}'\gg{\mathbf 0'}$. 
Hence by Corollary~\ref{cor:Jockusch-Stephan} the degree ${\mathbf a}$ is cohesive for $(A_i)_i$. 
We choose a realizer $G$ of $\COH$ that yields a set $r:=GK(p)\in2^\IN$ that is
cohesive for $(A_i)_i$ and of degree ${\mathbf a}$ and hence not cohesive. Since cohesive degrees are closed upwards with
respect to Turing reducibility by \cite[Corollary~1]{Joc73}, it follows that $H\langle p,GK(p)\rangle\leqT r$
is not of cohesive degree either, which is absurd.
\end{proof}

Next we prove that $\HYP$ is reducible to $\COH$.
The proof is based on a relativized version of the proof of \cite[Theorem~3.1]{JS93}. 

\begin{theorem}
\label{thm:HYP-COH}
$\HYP\leqSW\COH$.
\end{theorem}
\begin{proof}
Given $p\in\IN^\IN$ we compute a sequence $(R_i)_i$ of all primitive recursive sets relative to $p$.
Given some $A\in\COH((R_i)_i)$ we can compute the principal function $p_A$ of $A$ (i.e., $A=\{p_A(0)<p_A(1)<p_A(2)<...\}$).
We claim that $p_A\in\HYP(p)$. This yields the reduction $\HYP\leqSW\COH$.

We now follow a relativized version of the proof of \cite[Theorem~3.1]{JS93}. 
Let us assume for a contradiction that $p_A\not\in\HYP(p)$, so
there is some $r\leqT p$ that dominates $p_A$ in the sense that $p_A(n)\leq r(n)$ for all $n\in\IN$. 
If we can prove that every partial $p'$--computable function $\gamma:\In\IN\to\{0,1\}$ can be extended
to a $p'$--computable function $h:\IN\to\{0,1\}$, then we obtain $[p']\gg[p']$ by Proposition~\ref{prop:PA},
which is a contradiction to Proposition~\ref{prop:way-below}(1).
Let $\gamma:\In\IN\to\{0,1\}$ be $p'$--computable. Then by the limit lemma there is
a function $g:\IN^2\to\{0,1\}$ that is primitive recursive in $p$ and such that 
$\gamma(e)=\lim_{s\to\infty}g(e,s)$ for all $e\in\dom(\gamma)$. We note that $p_A(n)\in B_n:=\{n,n+1,....,r(n)\}$,
since $r(n)\geq p_A(n)\geq n$ for all $n$.
Let $S$ be the set of all pairs $(e,y)\in\IN\times\{0,1\}$ such that there are only finitely many $n$ with $g(e,s)=y$ for all $s\in B_n$.
Then $S$ is c.e.\ in $\langle p',r\rangle\leqT p'$. We prove that for each $e$ there exists $y\in\{0,1\}$ with $(e,y)\in S$.
Let us assume the contrary. Then there is some $e$ such that $(e,y)\not\in S$ for both $y\in\{0,1\}$.
Then the function $g_e:\IN\to\IN,s\mapsto g(e,s)$ assumes each value $y\in\{0,1\}$ infinitely often on $A$,
and hence the sets $S_e:=\{s:g(e,s)=1\}$ form a sequence of sets that are primitive recursive in $p$
and such that $A$ is not cohesive for $(S_e)_e$. This contradiction to $A\in\COH((R_i)_i)$ ensures that for each $e$ there exists $y\in\{0,1\}$ with $(e,y)\in S$. 
Let $f$ be a function that selects the first $y$ with $(e,y)\in S$ in a $p'$--computable enumeration of $S$.
Then $f(e)\not=\gamma(e)$ for all $e\in\dom(\gamma)$, and hence $h:=1-f$ is a total $p'$--computable extension of $\gamma$.  
\end{proof}

While cohesiveness computes the hyperimmunity problem, it does not compute the Kleene-Post theorem.\footnote{We thank Bj{\o}rn Kjos-Hanssen for pointing out the 
idea for the proof of Proposition~\ref{prop:KPT-COH}, see {\tt http://mathoverflow.net/questions/188596/is-below-every-cohesive-set-a-1-generic}, and we thank
a referee for suggesting a strengthening of our original statement.}

\begin{proposition}
\label{prop:KPT-COH}
$\KPT\nleqW\COH_+$.
\end{proposition}
\begin{proof}
By a result of Jockusch~\cite[Corollary~2]{Joc73}, every degree $a$ that is high (in the sense that $a'\geq0''$) contains a cohesive set.
By Cooper's jump inversion theorem~\cite[Theorem~1]{Coo73} there is a minimal high degree $a$. 
Let us now assume that $\KPT\leqW\COH_+$, and let $p$ be a computable input to $\KPT$.
Then from this input we can compute a sequence $(R_i)_i$ of c.e.\ sets $R_i\In\IN$ such that from any $X\in\COH_+(R_i)_i$ and $p$ we can compute two incomparable sets.
However, $\COH_+(R_i)_i$ contains a set $X$ of minimal high degree,
from which together with computable $p$ one cannot compute two incomparable sets. This contradiction completes the proof.
\end{proof}

Theorem~\ref{thm:HYP-COH} implies $1\dash\WGEN\leqW\COH$ by Corollary~\ref{cor:HYP-1WGEN}.
This result cannot be strengthened to $1\dash\GEN\leqW\COH$ by the following observation, which follows from
Propositions~\ref{prop:KPT-COH}, \ref{prop:KPT-MLR} and Corollary~\ref{cor:KPT-GEN}.

\begin{corollary}
\label{cor:GEN-MLR-COH}
$1\dash\GEN\nleqW\COH_+$ and $\MLR\nleqW\COH_+$.
\end{corollary}

It is clear that $\COH_+$ is densely realized. Hence we obtain $\ACC_\IN\nleqW\COH_+$
and, in particular, the following.

\begin{corollary}
\label{cor:DNCN-COH}
$\DNC_\IN\nleqW\COH_+$.
\end{corollary}

Finally, we prove that cohesiveness is limit computable.

\begin{proposition}
\label{prop:COH-lim}
$\COH\leqSW\lim$.
\end{proposition}
\begin{proof}
We prove $\COH\leqW\J$, from which the result follows since $\J\equivSW\lim$ are cylinders.
Given a sequence $(R_i)_i$ of sets $R_i\In\IN$ we
use the notation $1\cdot R:=R$ and $(-1)\cdot R:=\IN\setminus R$ for sets $R\In\IN$, and for every word $y\in\{0,1\}^*$ we define
\[R^y:=\bigcap_{i<|y|}(-1)^{y(i)}\cdot R_i.\]
We can use $\J$ in order to decide whether $|R^y|>|y|$ for any given $y\in\{0,1\}^*$, where $|R^y|$ denotes the cardinality of $R_y$ and $|y|$ the length of $y$.
This enables us to compute a sequence $(y_n)_n$ of words $y_n\in\{0,1\}^*$ such that $y_n$ is the lexicographically smallest word in $\{0,1\}^n$ such that
$|R^{y_n}|>|y_n|$ for all $n\in\IN$.
Such a sequence $(y_n)_n$ exists since there is a cohesive set $R$ for the sequence $(R_i)_i$, and $\lim_{n\to\infty}y_n(m)$ exists for every $m\in\IN$ since we choose the lexicographically
smallest $y_n$ for every $n$. Now we can also compute an injective sequence $(r_n)_n$ with $r_n\in R^{y_n}$ for all $n$ since $|R^{y_n}|>|y_n|=n$ and hence $R:=\{r_n:n\in\IN\}$
is cohesive for $(R_i)_i$. This algorithm yields an enumeration of an infinite cohesive set $R$ for $(R_i)_i$, and it is clear that from this enumeration
one can compute the characteristic function of an infinite subset of $R$, and any such subset is also cohesive for $(R_i)_i$.
\end{proof}

With the help of Proposition~\ref{prop:COH-COHD} we obtain the following corollary.

\begin{corollary}
\label{cor:COH+-lim}
$\COH_+\leqSW\lim'$.
\end{corollary}

\section{Cohesive Degrees}
\label{sec:COH-deg}

Corollary~\ref{cor:Jockusch-Stephan} highlights the relation of cohesive degrees to
PA--degrees (in light of Proposition~\ref{prop:PA}). A uniform version of this relation 
can be expressed as follows.

\begin{proposition}
\label{prop:COH-composition}
$[\COH]\equivSW\J_\DD^{-1}\circ\PA\circ\J_\DD$.
\end{proposition}
\begin{proof}
Firstly, we note that $\J_\DD^{-1}\circ\PA\circ\J_\DD({\bf b})=\{{\bf a}\in\DD:{\bf a}'\gg{\bf b}'\}$.
Hence, we obtain $[\COH]\leqSW\J_\DD^{-1}\circ\PA\circ\J_\DD$ by Corollary~\ref{cor:Jockusch-Stephan}.
For the other direction of the proof we note that given a degree ${\bf b}\in\DD$, we can
compute a sequence $(P_i)_i$ of all sets $P_i\In\IN$ that are primitive recursive relative to ${\bf b}$. 
By the proof of \cite[Theorem~2.1]{JS93} we obtain that every ${\bf a}\in[\COH](P_i)_i$
satisfies ${\bf a}'\gg{\bf b}'$, and hence we have that $\J_\DD^{-1}\circ\PA\circ\J_\DD\leqSW[\COH]$.
\end{proof}

Antitone jumps can be characterized as follows. We note that under the given conditions ${\mathbf a}\in\dom(f)$ and
${\mathbf a}\leqT{\mathbf b}$ imply that ${\mathbf b}\in\dom(f)$.

\begin{lemma}[Antitone jumps]
\label{lem:anti-jumps}
Let $f:\In\DD\mto\DD$ be antitone in the sense that
\[{\mathbf a}\leqT{\mathbf b}\TO\emptyset\not=f({\mathbf b})\In f({\mathbf a})\]
for all ${\mathbf a}\in\dom(f)$ and ${\mathbf b}\in\DD$. Then $f'\equivSW f\circ\J_\DD$.
\end{lemma}
\begin{proof}
Let $K:\In\IN^\IN\to\IN^\IN$ be a computable function with $\lim\circ K=\J$ and let $S$ be a realizer of $f'$.
Then we obtain for all $p\in\IN^\IN$ with $[p]\in\dom(f\circ\J_\DD)$
\[[SK(p)]\in f([\lim\circ K(p)])=f([p'])=f\circ \J_\DD([p]).\]
Hence, $SK$ is a realizer for $f\circ \J_\DD$ and $f\circ\J_\DD\leqSW f'$ follows.
For the other direction, let $R$ be a realizer of $f\circ\J_\DD$ and let $p\in\dom(f\circ[\lim])$.
Then $\lim p\leqT p'$ and since $f$ is antitone we obtain $[p']=\J_\DD([p])\in\dom(f)$ and 
\[[R(p)]\in f\circ \J_\DD([p])=f([p'])\In f([\lim p ]).\]
Thus $R$ is a realizer for $f'$ and $f'\leqSW f\circ\J_\DD$ follows.
\end{proof}

Using the previous observations we derive the following purely algebraic characterization of $\COH$ on degrees
in terms of the jump of Peano arithmetic.

\begin{theorem}[Cohesive degrees]
\label{thm:COH-degrees}
$[\COH]\equivW(\lim\to\PA')\equivW(\J_\DD\to\PA')$.
\end{theorem}
\begin{proof}
Since $\J_\DD\leqSW\J\equivSW\lim$ and $\PA'\equivSW\PA\circ\J_\DD$ by Lemma~\ref{lem:anti-jumps},
it suffices to prove $(\J_\DD\to\PA\circ\J_\DD)\leqW[\COH]\leqW(\J\to\PA\circ\J_\DD)$.
It is clear that $\PA\circ\J_\DD\leqW\J_\DD*(\J_\DD^{-1}\circ\PA\circ\J_\DD)\equivW\J_\DD*[\COH]$ by Proposition~\ref{prop:COH-composition}.
Hence $(\J_\DD\to\PA\circ\J_\DD)\leqW[\COH]$. For the second reduction, suppose
$h$ is such that $\PA\circ\J_\DD\leqW\J*h$. Without loss of generality,
we can assume that $h$ is of type $h:\In\IN^\IN\mto\IN^\IN$ and that $h$ is
a cylinder (since the cylindrification $\id\times h$ satisfies $\id\times h\equivW h$).
Since $\J$ is a cylinder too, by \cite[Lemma~21]{BP16} there is a computable $G$ such that $\J*h\equivW \J\circ G\circ h$.
Since $\J\circ G$ is limit computable, there is a computable $H$ such that 
$\J\circ G=H\circ\J$. This implies that $\PA\circ\J_\DD\leqW H\circ\J\circ h$.
Hence there is a computable $K$ such that for each $p\in\IN^\IN$ we obtain that $q:=\J\circ h\circ K(p)$ satisfies
$[q]\in\PA\circ\J_\DD([p])$, which implies $[q]\gg[p']=[p]'$.
Thus $[h\circ K(p)]'\gg [p]'$
by Proposition~\ref{prop:way-below}(4).
This implies, by Proposition~\ref{prop:COH-composition}, that $[\COH]\equivSW\J_\DD^{-1}\circ\PA\circ\J_\DD\leqW h$,
and hence we obtain $[\COH]\leqW(\J\to\PA\circ\J_\DD)$.
\end{proof}

Since $\PA\equivSW[\WKL]$ by Corollaries~\ref{cor:PA-DNC} and \ref{cor:WKL-DNC}, 
we can also express Theorem~\ref{thm:COH-degrees} as follows.

\begin{corollary}
\label{cor:COH-WKL-degrees}
$[\COH]\equivW(\lim\to[\WKL]')$. 
\end{corollary}

In the next section we will prove a corresponding characterization of $\COH$.

\section{The Weak Bolzano-Weierstra\ss{} Theorem}
\label{sec:BWT}

In this section we briefly discuss the relation of the Bolzano-Weierstra\ss{} theorem to the
cohesiveness problem, which was already considered \cite{Kre11}.
We recall that the Bolzano-Weierstra\ss{} theorem can be formalized as follows (see \cite{BGM12,Kre12}).

\begin{definition}[Bolzano-Weierstra\ss{} theorem]
Let $X$ be a computable metric space and let
\[\BWT_X:\In X^\IN\mto X,(x_i)_i\mapsto\{x:x\mbox{ is a cluster point of }(x_i)_i\}\]
where $\dom(\BWT_X)$ is the set of all sequences $(x_i)_i$  such that $\{x_i:i\in\IN\}$ has compact closure.
\end{definition}

We note the following \cite[Corollaries~11.6 and 11.17]{BGM12}.

\begin{fact}
\label{fact:BWT}
$\BWT_{2^\IN}\equivSW\BWT_\IR\equivSW\BWT_{[0,1]}\equivSW\BWT_{[0,1]^\IN}\equivSW\WKL'$.
\end{fact}

$\BWT_\IR$ determines a cluster point of a given bounded sequence $(x_n)_n$.
One could also consider the problem of finding a convergent subsequence of $(x_n)_n$.
Equivalently, we define the following weakening of the Bolzano-Weierstra\ss{} theorem.

\begin{definition}[Weak Bolzano-Weierstra\ss{} theorem]
Let $X$ be a computable metric space with Cauchy representation $\delta_X$.
By $\WBWT_X:\In X^\IN\mto X'$ we denote the same problem as $\BWT_X$
but with the jump $\delta_X'$ of the Cauchy representation as representation on the output side.
\end{definition}

This means that the output of a realizer of $\WBWT_X$ on some input $(x_n)_n$ is a sequence in $\IN^\IN$
that converges to a Cauchy name of a cluster point of $(x_n)_n$. 
The following result expresses the relation between the Bolzano-Weierstra\ss{} theorem
and the weak Bolzano-Weierstra\ss{} theorem.

\begin{theorem}[Weak Bolzano-Weierstra\ss{} theorem]
\label{thm:WBWT}
$\WBWT_X\equivW(\lim\to\BWT_X)$ for every computable metric space $X$.
\end{theorem}
\begin{proof}
It follows immediately from the definitions that $\BWT_X\leqW\lim*\WBWT_X$, which implies
$(\lim\to\BWT_X)\leqW\WBWT_X$.
Let now $h$ be a multi-valued function such that $\BWT_X\leqW\lim*h$. Without loss of generality we can
assume that $h$ is of type $h:\In\IN^\IN\mto\IN^\IN$ and that $h$ is a cylinder.
Hence there are computable functions $H,K$ such that $\lim\circ H\circ h\circ K\equivW\lim*h$.
Since $\BWT_X\leqW\lim*h\equivW\lim\circ H\circ h\circ K$, it follows that
$\WBWT_X\leqW H\circ h\circ K\leqW h$, which implies $\WBWT_X\leqW(\lim\to\BWT_X)$.
\end{proof}

We obtain $\WBWT_{2^\IN}\equivW\WBWT_\IR\equivW\WBWT_{[0,1]}\equivW\WBWT_{[0,1]^\IN}$
with Fact~\ref{fact:BWT} and Theorem~\ref{thm:WBWT}. 
In Proposition~\ref{prop:SBWT} we are going to prove a slightly stronger result.

While the weak Bolzano-Weierstra\ss{} theorem determines a sequence that converges
to a cluster point, one could also consider a variant where the result is a function that
selects a converging subsequence.

\begin{definition}[Subsequential Bolzano-Weierstra\ss{} theorem]
Let $X$ be a computable metric space. We define $\SBWT_X:\In X^\IN\mto\IN^\IN$ by
\[\SBWT_X((x_i)_i):=\{s\in\IN^\IN:\mbox{$(x_{s(n)})_n$ converges and $s$ is strictly monotone}\}\]
where $\dom(\SBWT_X)$ is the set of all sequences $(x_i)_i$ such that $\{x_i:i\in\IN\}$ has compact closure.
\end{definition}

This version comes closest to the weak Bolzano-Weierstra\ss{} theorem that has been introduced in \cite{Kre11}.

\begin{proposition}
\label{prop:WBWT-SBWT}
$\WBWT_X\equivW\SBWT_X$ for every complete computable metric space $X$.
\end{proposition}
\begin{proof}
Given a sequence $(x_i)_i$, we can use $\SBWT_X(x_i)_i$ in order to find a strictly monotone $s\in\IN^\IN$ such that
$(x_{s(n)})_n$ converges. Hence $\lim_{n\to\infty} x_{s(n)}$ is a cluster point of $(x_i)_i$.
Since $\lim_X\leqSW\lim$, e.g., by \cite[Proposition~9.1]{Bra05}, we obtain $\WBWT_X\leqW\SBWT_X$.
On the other hand, given $(x_i)_i$, we can use $\WBWT_X(x_i)_i$ in order to find a sequence of names $(p_n)_n$ such that $p:=\lim_{n\to\infty}p_n$
is a name for a cluster point $x$ of $(x_i)_i$. Since $(X,d)$ is a complete computable metric space, we can assume without loss of generality that 
we use a total Cauchy representation (this exists, for example, by \cite[Lemma~6.1]{BHK16}). Hence there is a sequence $(y_n)_n$ in $X$
such that $p_n$ is a name for $y_n$ and $p$ is a name for $x=\lim_{n\to\infty}y_n$. 
Now, for every $n\in\IN$ we can find a number $s(n)\in\IN$ such that $d(y_n,x_{s(n)})<2^{-n}$, which exists since $(y_n)_n$ converges to a cluster point of $(x_i)_i$.
Additionally, we can construct such an $s$ that is strictly increasing.
It follows that $\lim_{n\to\infty}x_{s(n)}=x$, and the algorithm yields $\SBWT_X\leqW\WBWT_X$.
\end{proof}

We note that this proof only yields Weihrauch equivalence and not strong Weih\-rauch equivalence,
as the access to the input is crucial in both directions.

As for the other versions of the Bolzano-Weierstra\ss{} theorem, $\SBWT_X$ yields
one and the same equivalence class for many different computable metric spaces $X$.

\begin{proposition}
\label{prop:SBWT}
$\SBWT_{2^\IN}\equivSW\SBWT_\IR\equivSW\SBWT_{[0,1]}\equivSW\SBWT_{[0,1]^\IN}$.
\end{proposition}
\begin{proof}
The reduction $\SBWT_{2^\IN}\leqSW\SBWT_\IR$ can be established using the function
$f:2^\IN\to\IR,p\mapsto\sum_{i=0}^\infty 2p(i)3^{-i-1}$ that maps $2^\IN$ computably and injectively to the Cantor discontinuum.
This map has a partial continuous inverse, and hence $(x_{s(n)})_n$ converges in $2^\IN$ whenever $(f(x_{s(n)}))_n$ converges in $\IR$.
The inverse reduction, $\SBWT_{\IR}\leqSW\SBWT_{[0,1]}$, follows similarly using the computable\footnote{That $f$ is computable follows since $\arctan$ is computable,
which holds since $\arctan(x)=\int_0^x\frac{1}{1+t^2}\;{\mathrm d}t$, and integration is computable by \cite[Theorem~6.4.1]{Wei00}.}
function $f:\IR\to[0,1],x\mapsto\frac{1}{\pi}\arctan(x)+\frac{1}{2}$. 
The reduction $\SBWT_{[0,1]}\leqSW\SBWT_{[0,1]^\IN}$ follows using the canonical injection $f:[0,1]\to[0,1]^\IN$.
Finally, for the reduction $\SBWT_{[0,1]^\IN}\leqSW\SBWT_{2^\IN}$ 
we can assume, without loss of generality, that $[0,1]^\IN$ is represented by a total version of the Cauchy representation $\rho:2^\IN\to[0,1]^\IN$ (which exists by \cite[Proposition~4.1]{BBP12}
since $[0,1]^\IN$ is computably compact).
Given a sequence $(x_n)_n$ in $[0,1]^\IN$ as a sequence of names $p_n\in2^\IN$ with $x_n=\rho(p_n)$,
we obtain that $(x_{s(n)})_n$ converges if $(p_{s(n)})_n$ converges. 
\end{proof}

We note that basically the same proof shows that we could also get strong equivalences in the case of $\WBWT$.
Now essentially the same proof as that of \cite[Theorem~3.2]{Kre11} yields the following result.

\begin{theorem}[Subsequential Bolzano-Weierstra\ss{} theorem]
\label{thm:SBWT-COH}
$\SBWT_\IR\equivSW\COH$.
\end{theorem}
\begin{proof}
Firstly, we note that $\SBWT_{[0,1]}\equivSW\SBWT_{[0,1]}|_{\IQ^\IN}$.
This is because given a sequence $(x_i)_i$ in $[0,1]$,
we can compute a sequence $(y_i)_i$ in $\IQ\cap[0,1]$ such that $|x_i-y_i|<2^{-i}$, and hence 
the sets of cluster points of both sequences coincide.
By Proposition~\ref{prop:SBWT} it suffices to prove $\SBWT_{[0,1]}|_{\IQ^\IN}\leqSW\COH\leqSW\SBWT_{2^\IN}$.

Now we prove $\SBWT_{[0,1]}|_{\IQ^\IN}\leqSW\COH$. Given a sequence $(x_i)_i$ of rational
numbers in $[0,1]$ we compute the sequence $(R_i)_{i}$ of sets with
\[R_i:=\left\{j\in\IN:x_j\in\bigcup_{k<2^{i}}\left[\frac{2k}{2^{i+1}},\frac{2k+1}{2^{i+1}}\right]\right\}\]
for all $i\in\IN$. 
Now given a cohesive set $A\in\COH((R_i)_i)$ for $(R_i)_i$ we can compute the principal function
$p_A$ of $A$, which is strictly monotone. It suffices to show that $(x_{p_A(n)})_n$ is a Cauchy sequence.
For every set $R\In\IN$ we use the notation $1\cdot R:=R$ and $(-1)\cdot R:=\IN\setminus R$,
and we define for every word $y\in\{0,1\}^*$
\[R^y:=\bigcap_{i<|y|}(-1)^{y(i)}\cdot R_{i}.\]
By the definition of the sets $R_i$ it follows that $i,j\in R^y$ implies $|x_i-x_j|\leq2^{-|y|+1}$.
Since $A$ is cohesive for $(R_i)_i$ it follows that for every $k\in\IN$ there is some $m\in\IN$
and some $y\in\{0,1\}^{k+2}$ such that $p_A(n)\in R^y$ for all $n>m$. In particular, for every $k\in\IN$
there is some $m\in\IN$ such that $|x_{p_A(n)}-x_{p_A(m)}|<2^{-k}$ for all $n>m$. This means that $(x_{p_A(n)})_n$
is a Cauchy sequence.

Now we prove that $\COH\leqSW\SBWT_{2^\IN}$. Given a sequence $(R_i)_i$ of sets $R_i\In\IN$
we compute a sequence $(x_i)_i$ in $2^\IN$ by 
\[x_i(n):=\left\{\begin{array}{ll}
  1 & \mbox{if $i\in R_n$}\\
  0 & \mbox{otherwise}
\end{array}\right..\]
Now $\SBWT_{2^\IN}((x_i)_i)$ yields a strictly increasing function $s:\IN\to\IN$ such that
$(x_{s(i)})_i$ is a convergent subsequence of $(x_i)_i$. Given $s$ we can compute its range $A:=\{s(n):n\in\IN\}$,
and it suffices to show that $A$ is cohesive for $(R_i)_i$. 
The fact that $(x_{s(i)})_i$ is convergent, and hence Cauchy,
implies that for all $k\in\IN$ there is some $m\in\IN$ such that $x_{s(j)}|_k=x_{s(m)}|_k$ for all $j\geq m$;
i.e., such that $s(j)\in R_n\iff s(m)\in R_n$ for all $j\geq m$ and $n<k$.
This proves that $A$ is cohesive for $(R_i)_i$.
\end{proof}

Using Theorems~\ref{thm:WBWT} and \ref{thm:SBWT-COH}, Proposition~\ref{prop:WBWT-SBWT} and Fact~\ref{fact:BWT} 
we get the following purely algebraic characterization of cohesiveness
in terms of the jump of weak K\H{o}nig's lemma, which is the counterpart of Corollary~\ref{cor:COH-WKL-degrees}
that expresses a similar relation on the corresponding problems on Turing degrees.

\begin{corollary}[Cohesiveness]
\label{cor:COH}
$\COH\equivW(\lim\to\WKL')$.
\end{corollary}

We can derive the following interesting consequence of this corollary.

\begin{proposition}
\label{cor:COR-WBWT2}
$\COH\equivW\widehat{\WBWT_2}$.
\end{proposition}
\begin{proof}
It is clear that $\WBWT_2\leqW\WBWT_{2^\IN}\leqW\COH$ by Theorem~\ref{thm:SBWT-COH} and Proposition~\ref{prop:WBWT-SBWT}.
Since $\COH$ is parallelizable by definition, it follows that $\widehat{\WBWT_2}\leqW\COH$. For the other direction of the reduction we note that
$\BWT_{2^\IN}\equivW\widehat{\BWT_2}$ by \cite[Corollary~11.12]{BGM12}.
By \cite[Proposition~41]{BP16} we obtain for $F:=\lim_2*\WBWT_2$
that $\widehat{F}\leqW\widehat{\lim_2}*\widehat{\WBWT_2}$.
Hence it follows from Fact~\ref{fact:BWT} that
\[\WKL'\equivW\BWT_{2^\IN}\leqW\widehat{\BWT_{2}}\leqW\widehat{F}\leqW\widehat{\lim\nolimits_2}*\widehat{\WBWT_2}\leqW\lim*\widehat{\WBWT_2}.\]
Hence $\COH\leqW\widehat{\WBWT_2}$ by Corollary~\ref{cor:COH}.
\end{proof}

From Corollary~\ref{cor:COH} it follows that $\WKL'\leqW\lim*\COH$. In Corollary~\ref{cor:WKL-lim-COH} we are going to 
prove that even equivalence holds. As a preparation,
we first prove a theorem that shows that computable functions that map converging
sequences to converging sequences can be mimicked on the limits by a function that is
computable in the halting problem.

We recall that we use the coding $\langle\;\rangle:(\IN^\IN)^\IN\to\IN^\IN$ given by
\[\langle p_0,p_1,p_2,...\rangle\langle n,k\rangle:=p_n(k)\]
for all $p_i\in\IN^\IN$ and $n,k\in\IN$. Given words $v_0,...,v_t\in\IN^*$, we define analogously
$\langle v_0,...,v_t\rangle$ to be the longest word $w\in\IN^*$ such that $w\langle n,k\rangle=v_n(k)$ is defined for
all $n,k\in\IN$ with $\langle n,k\rangle<|w|$.

\begin{theorem}[Double Limit]
\label{thm:double-limit}
For every computable function $G:\In\IN^\IN\to\IN^\IN$ there exists a function
$F:\In\IN^\IN\to\IN^\IN$ that is computable in $\emptyset'$ and such that
\[F(p)\in\lim\circ G\circ\lim\nolimits^{-1}(p)\]
for all $p\in\dom(\lim\circ G\circ\lim\nolimits^{-1})$. 
\end{theorem}
\begin{proof}
Using a computable standard embedding $\iota:\IN^\IN\into2^\IN$ we can assume without loss of generality
that $G$ is a computable function of type $G:\In\IN^\IN\to2^\IN$. Then there exists a computable monotone function $g:\IN^*\to2^*$
that approximates $G$ in the sense that $G(q)=\sup_{w\prefix q}g(w)$ for all $q\in\dom(G)$.
Let $g_i:\IN^*\to2^*$ be the part of $g$ that contributes to $(\lim\circ G(q))(i)$:
for $w\in\IN^*$ and $i\in\IN$ we define
$g_i(w):=g(w)\langle 0,i\rangle...g(w)\langle n,i\rangle$
for the largest $n\in\IN$ such that all the values are defined.
Let  $A\In\{0,1\}\times\IN^3\times(\IN^*)^2$  be the set of values $(b,k,i,s,w,\langle u_0,...,u_s\rangle)$
such that there exist $t>s$ and $u_{s+1},...,u_t\in\IN^*$ with 
\begin{enumerate}
\item[(a)] $w\prefix u_\iota$ for all $\iota=s+1,...,t$,
\item[(b)] $g_i(\langle u_00^t,...,u_s0^t,u_{s+1},...,u_t\rangle)$ contains at least $k$--times the bit $b$.
\end{enumerate} 
The set $A$ is c.e.\ and hence computable in $\emptyset'$.

Now we describe how we can compute a suitable function $F:\In\IN^\IN\to2^\IN$ with the help of $A$.
Given some input $p\in\dom(\lim\circ G\circ\lim\nolimits^{-1})$ we compute
$F(p)(i)$ with the help of a sequence $(u_s)_s$ of words $u_s\in\IN^*$. 

For each fixed $i\in\IN$ we determine this sequence $(u_s)_s$ inductively in stages $j=0,1,2,...$.
We start with the empty sequence $(u_s)_s$ and $s_0:=-1$.
At stage $j=2k+b$, $u_0,...,u_{s_j}$ are already determined and we check whether
\begin{eqnarray}
\label{eq:double-limit}
(b,k,i,s_j,p|_j,\langle u_0,...,u_{s_j}\rangle)&\in& A.
\end{eqnarray}
If so, then we compute corresponding words $u_{s_j+1},...,u_t$
that satisfy the conditions (a) and (b) given above and we extend $u_0,...,u_{s_j}$ by these words, so $s_{j+1}:=t$.
Otherwise, we leave the sequence as it is and set $s_{j+1}:=s_j$. 

For each $k\in\IN$ the test~(\ref{eq:double-limit}) above is positive for at least one $b\in\{0,1\}$: the word $\langle u_0,...,u_{s_j}\rangle$ can be extended to 
$q:=\langle u_0\widehat{0},...,u_{s_j}\widehat{0},p,p,p,...\rangle$ and $\lim(q)=p$, and hence $q\in\dom(G)$
and $\lim\circ G(q)$ exists. This means that $(\lim\circ G(q))(i)=b$ for some $b\in\{0,1\}$, and for this $b$ a suitable $t$ and an
extension $u_{s_j+1},...,u_t$ can be found by the computability of $G$.
In particular, the sequence $(u_s)_s$ is actually infinite, and $u:=\langle u_0\widehat{0},u_1\widehat{0},u_2\widehat{0},...\rangle$ 
satisfies $\lim(u)=p$ by construction. 

Since $\lim(u)=p$, it follows that $u\in\dom(G)$ and $\lim\circ G(u)$ exists, 
which means that $(\lim\circ G(u))(i)=d$ for some $d\in\{0,1\}$.
Hence, for our fixed $(u_s)_s$ and $b=1-d$, the test~(\ref{eq:double-limit}) is positive for 
only finitely many $k$ and $j=2k+b$. If the test is negative for some fixed $k,b$, then it is also
negative for all larger $k$ and the corresponding $j$.
Hence there is a minimal $k\in\IN$ such that the test~(\ref{eq:double-limit}) 
is positive for $j=2k+b$ with a unique $b\in\{0,1\}$, and for this unique $b$
we obtain $b=d$.

Hence, in order to compute $F(p)(i)$ given $u$, we search for the minimal $k\in\IN$
with the property that (\ref{eq:double-limit}) is satisfied for $j=2k+b$ with
a unique $b$ and we let $F(p)(i)=b$, which guarantees 
\[F\circ\lim(u)=F(p)=\lim\circ G(u)\]
and hence $F(p)\in\lim\circ G\circ\lim\nolimits^{-1}(p)$.
Since $A$ is computable in the halting problem $\emptyset'$, it follows that $F$ is also.
\end{proof}

If the function $G$ is extensional in the sense that the limit of its output only depends on the limit
of its input, then the function $F$ is uniquely determined. This unique version 
generalizes \cite[Theorem~22 (2)$\TO$(1)]{BH02}. Already this unique version shows
that $F$ cannot be computable in general (a function $G$ that is constant and computes a
sequence that converges to the halting problem is a counterexample). 
The technique used to prove Theorem~\ref{thm:double-limit} is somewhat reminiscent of the
first jump control technique~\cite[Section~4]{CJS01}.
We obtain the following corollary.
We note that $f\leqSW\cc_{\emptyset'}\times\id$ holds if and only if $f$ is computable with
respect to the halting problem.  

\begin{corollary}
$\lim\stars\lim^{-1}\equivSW \cc_{\emptyset'}\times\id$ and $\lim*\lim^{-1}\equivW\lim$.
\end{corollary}
\begin{proof}
We show that the maximum of $M:=\{f_0\circ f_1:f_0\leqSW\lim,f_1\leqSW\lim^{-1}\}$ with respect to $\leqSW$ exists and is 
strongly Weihrauch equivalent to $\cc_{\emptyset'}\times\id$.
Let $f_0\leqSW\lim$ and $f_1\leqSW\lim^{-1}$. Then there are computable $H_0,K_0,H_1,K_1$ such that
$H_0\circ\lim\circ K_0(p)\in f_0(p)$ for all $p\in\dom(f_0)$ and $\emptyset\not=H_1\circ\lim^{-1}\circ K_1(p)\In f_1(p)$ for all $p\in\dom(f_1)$.
Let $G:=K_0\circ H_1$ and
let $F\leqSW\cc_{\emptyset'}\times\id$ be the function from Theorem~\ref{thm:double-limit} for this $G$.
We obtain
\[f_0\circ f_1=H_0\circ\lim\circ G\circ\lim\nolimits^{-1}\circ K_1\leqSW\lim\circ G\circ\lim\nolimits^{-1}\leqSW F\leqSW \cc_{\emptyset'}\times\id.\]
On the other hand, there are clearly computable functions $H,G,K$ such that $\cc_{\emptyset'}\times\id=H\circ\lim\circ G\circ\lim^{-1}\circ K\in M$. 
Hence $\lim\stars\lim^{-1}$ exists and $\lim\stars\lim^{-1}\equivSW\cc_{\emptyset'}\times\id$. 

The second equivalence holds since $\lim^{-1}$ is computable.   
\end{proof}

The proof of Theorem~\ref{thm:double-limit} works analogously for the following parameterized
version. We just need to replace the c.e.\ set $A$ in the proof by a suitable set $A_q$ that depends on the
additional parameter $q$. In this case $A_q$ is c.e.\ in $q$, and hence it can be computed with the help of the Turing jump $\J(q)=q'$ of $q$.

\begin{theorem}[Parameterized Double Limit]
\label{thm:para-double-limit}
For every computable function $G:\In\IN^\IN\to\IN^\IN$ there exists a computable function
$F:\In\IN^\IN\to\IN^\IN$ such that
\[F\langle \J(q),p\rangle\in\lim\circ G\circ\langle\id\times\lim\nolimits^{-1}\rangle(q,p)\]
for all $(q,p)\in\dom(\lim\circ G\circ\langle\id\times\lim\nolimits^{-1}\rangle)$. 
\end{theorem}

Now we obtain the following characterization of the Bolzano-Weierstra\ss{} theorem
with the help of the parameterized double limit theorem~\ref{thm:para-double-limit}.

\begin{theorem}
\label{thm:BWT-COH}
$\BWT_\IR\equivW\lim*\WBWT_\IR$ and $\BWT_\IR\equivSW\lim\stars\WBWT_\IR$.
\end{theorem}
\begin{proof}
It is clear that $\BWT_\IR\leqW\lim*\WBWT_\IR$.
For the other reduction of the first claim it suffices to prove $\lim_{2^\IN}*\WBWT_{2^\IN}\leqW\BWT_{2^\IN}$
by Fact~\ref{fact:BWT}, Propositions~\ref{prop:SBWT}, \ref{prop:WBWT-SBWT} and \cite[Proposition~9.1]{Bra05}. 
To this end, let $g\leqW\lim_{2^\IN}$ and $h\leqW\WBWT_{2^\IN}$ be such that
$g\circ h$ exists. Since $\lim_{2^\IN}$ is a cylinder, we can even assume $g\leqSW\lim_{2^\IN}$.
Without loss of generality, we can assume that
$g,h$ are of type $g,h:\In\IN^\IN\mto\IN^\IN$.
We need to prove that $g\circ h\leqW\BWT_{2^\IN}$.
There are computable $H,G,K:\In\IN^\IN\to\IN^\IN$ such that
\[H\circ\lim\nolimits_{2^\IN}\circ G\langle q,\WBWT_{2^\IN}K(q)\rangle\In g\circ h(q)\]
for all $q\in\dom(h)$. Since $\WBWT_{2^\IN}K(q)$ yields arbitrary converging sequences $p\in2^\IN$ as output that converge to some cluster point of the input $K(q)$, $G$
needs to be defined on all corresponding pairs $\langle q,p\rangle$ for each fixed $q\in\dom(h)$.
By the parameterized double limit theorem (Theorem~\ref{thm:para-double-limit}) we obtain that there is a computable function $F:\In\IN^\IN\to2^\IN$ such that
\[H\circ F\langle\J(q),\BWT_{2^\IN}K(q)\rangle\In H\circ\lim\nolimits_{2^\IN}\circ G\langle q,\lim\nolimits^{-1}\circ\BWT_{2^\IN}K(q)\rangle\In g\circ h(q)\]
for all $q\in\dom(h)$, which implies $g\circ h\leqW\J\times\BWT_{2^\IN}\equivW\lim\times\BWT_{2^\IN}$,
since $\J\equivW\lim$.
Since $\BWT_{2^\IN}$ is idempotent \cite[Corollary~11.13]{BGM12} and $\lim\leqW\BWT_{2^\IN}$ \cite[Corollary~11.22]{BGM12}, we have that
$\lim\times\BWT_{2^\IN}\leqW\BWT_{2^\IN}$, and hence
$g\circ h\leqW\BWT_{2^\IN}$.

We now show that the maximum of 
\[M:=\{f_0\circ g_0:f_0\leqSW\lim,g_0\leqSW\WBWT_\IR\}\] 
with respect to $\leqSW$ exists and is strongly Weihrauch equivalent to $\BWT_\IR$.
Clearly $\BWT_\IR\in M$. On the other hand, $h\in M$ implies
\[h\leqW\lim*\WBWT_\IR\leqW\BWT_\IR.\]
Since $\BWT_\IR$ is a cylinder \cite[Corollary~11.13]{BGM12}, we obtain  
$h\leqSW\BWT_\IR$. Altogether, this proves the claim.
\end{proof}

By Fact~\ref{fact:BWT}, Proposition~\ref{prop:WBWT-SBWT}, Theorem~\ref{thm:SBWT-COH}
and \ref{thm:BWT-COH}
we obtain the following corollary.

\begin{corollary}
\label{cor:WKL-lim-COH}
$\WKL'\equivW\lim*\COH$.
\end{corollary}

We continue this section with some lowness properties of $\WBWT_\IR$ and cohesiveness $\COH$.
In \cite{BGM12} a multi-valued function $f$ is called {\em low$_n$}, for $n\geq1$, if $f\leqSW\Low_n$, where $\Low_n:=(\J^{-1})^{\circ n}\circ\lim^{\circ n}$.
Here $f^{\circ n}$ denotes the $n$--fold composition of $f$ with itself, i.e., $f^{\circ 1}=f$, $f^{\circ 2}=f\circ f$ and so forth. 
Instead of low$_1$ we simply say {\em low}.  We are going to use the following characterization of lowness.

\begin{proposition}[Lowness]
\label{prop:lowness}
Let $f$ and $n\geq1$ be such that $\lim^{\circ n}\stars f$ exists. Then
$f$ is low$_n$ if and only if $\lim^{\circ n}\stars f\leqSW \lim^{\circ n}$.
\end{proposition}
\begin{proof}
It suffices to consider $f,g$ of type $f,g:\In\IN^\IN\mto\IN^\IN$.
Theorem~8.6 of \cite{BGM12} provides the following characterization of lowness for all $n\geq1$:
\[\mbox{$f$ is low$_n\iff (\forall g\leqW\lim^{\circ n})\;g\circ f\leqW\lim^{\circ n}$.}\]
We note that $\leqW$ can be replaced by $\leqSW$ here, since $\lim^{\circ n}$ is a cylinder.
For the if direction, we assume that $\lim^{\circ n}\stars f\leqSW \lim^{\circ n}$, and we consider $g\leqSW\lim^{\circ n}$.
Then we directly obtain $g\circ f\leqSW\lim^{\circ n}\stars f\leqSW \lim^{\circ n}$.
For the only if direction, we assume that $f$ is low$_n$ and that $g_0\leqSW\lim^{\circ n}$ and $f_0\leqSW f$ hold.
Then there are computable $H_0,K_0,H_1,K_1$ such that 
\[g_0\circ f_0\leqSW H_0\circ\lim\nolimits^{\circ n}\circ K_0\circ H_1\circ f\circ K_1\leqSW\lim\nolimits^{\circ n}\circ K_0\circ H_1\circ f\leqSW\lim\nolimits^{\circ n},\]
since $g:=\lim^{\circ n}\circ K_0\circ H_1\leqSW\lim^{\circ n}$.
\end{proof}

\begin{theorem}
\label{thm:WBWT-low}
$\WBWT_\IR\leqSW\Low_2$ and $\WBWT_\IR\nleqSW\Low$, i.e., $\WBWT_\IR$ is low$_2$, but not low.
\end{theorem}
\begin{proof}
Corollary~11.15 of \cite{BGM12} says that $\BWT_\IR\leqSW\Low'\equivSW\Low\stars\lim$, and 
Corollary~8.8 of \cite{BGM12} says that $\lim\stars\Low\equivSW\lim$. In particular, $\lim\stars\Low$ exists as a maximum
since $\lim\equivSW\J\circ\J^{-1}\circ\lim=\J\circ\Low$.
By Theorem~\ref{thm:BWT-COH}, $\lim\stars\WBWT_\IR\equivSW\BWT_\IR$ exists, and it is a cylinder.
With the help of Lemmas~\ref{lem:strong-composition} and \ref{lem:associative} we obtain
\[\lim\stars\lim\stars\WBWT_\IR\leqSW\lim\stars\BWT_\IR\leqSW\lim\stars\Low\stars\lim\leqSW\lim\stars\lim.\]
This implies $\WBWT_\IR\leqSW\Low_2$ by Proposition~\ref{prop:lowness}.

Let us now assume that $\WBWT_\IR\leqSW\Low$. With the help of Theorem~\ref{thm:BWT-COH} we obtain
\[\BWT_\IR\leqSW\lim\stars\WBWT_\IR\leqSW\lim\stars\Low\leqSW\lim\]
in contradiction to \cite[Theorem~12.7]{BGM12}, which states that $\lim\lW\BWT_\IR$.
\end{proof}

The first statement of this result implies the following non-uniform version, which was already
proved in \cite[Theorem~3.5(1)]{Kre11}.

\begin{corollary}
For every bounded sequence $(x_n)_n$ of real numbers there exists a low$_2$ sequence of reals that converges
to a cluster point of $(x_n)_n$.
\end{corollary}

The proof idea of the second part of Theorem~\ref{thm:WBWT-low} can be used to show that low$_2$ cannot be replaced by low in this corollary.

We note that Proposition~\ref{prop:WBWT-SBWT} is only formulated for ordinary Weihrauch reducibility
and not for strong Weihrauch reducibility. Hence we cannot directly transfer Theorem~\ref{thm:WBWT-low} to $\SBWT_\IR$. 
However, we can easily derive a corresponding result for $\SBWT_\IR$ or equivalently (by Theorem~\ref{thm:SBWT-COH}) for $\COH$. 

\begin{theorem}
\label{thm:COH-low}
$\COH\leqSW\Low_2$ and $\COH\nleqW\Low$, i.e., $\COH$ is low$_2$, but not low.
\end{theorem}
\begin{proof}
By Proposition~\ref{prop:lowness}, $\lim*\WKL\leqW\lim$ since $\WKL$ is low by \cite[Corollary~8.5]{BBP12}
and hence $\lim*\WKL'\leqW\lim'$.
With Corollary~\ref{cor:WKL-lim-COH} we obtain
\[\lim\nolimits'*\COH\leqW\lim*\lim*\COH\leqW\lim*\WKL'\leqW\lim*\lim\leqW\lim\nolimits'.\]
Hence $\COH$ is low$_2$ by Proposition~\ref{prop:lowness}.

Let us now assume that $\COH\leqW\Low$. With the help of Theorems~\ref{thm:BWT-COH}, \ref{thm:SBWT-COH} and Proposition~\ref{prop:WBWT-SBWT} we obtain
\[\BWT_\IR\leqW\lim*\WBWT_\IR\leqW\lim*\COH\leqW\lim*\Low\leqW\lim\]
in contradiction to \cite[Theorem~12.7]{BGM12}, which states that $\lim\lW\BWT_\IR$.
\end{proof}

Finally, we want to show that $\COH$ is not probabilistic,
for this purpose we need the following technical lemma.

\begin{lemma}
\label{lem:jump-measure}
The set $C=\{n\in\IN:\mu\{p\in2^\IN:F(p)(n)=j\}>r\}$ is a $\SO{2}$--set for every limit computable $F:2^\IN\to2^\IN$, $r\in\IQ$ and $j\in\IN$.
\end{lemma}
\begin{proof}
By $V_n:=\{q\in2^\IN:q(n)=j\}$ we define a computable sequence of c.e.\ open sets, and hence $A_n:=F^{-1}(V_n)$ is a computable
sequence of $\SO{2}$--sets with respect to the effective Borel hierarchy \cite{Bra05}.
For every such $\SO{2}$--set there is a computable  sequence $(U_i)_i$ of open sets such that $A_n=(\bigcap_{i=0}^\infty U_i)^{\rm c}$.
By $U_i[s]$ we denote the enumeration of $U_i$ up to stage $s\in\IN$ (which is a finite union of basic open balls).
Then we obtain, by the continuity of the probability measure $\mu$,
\[\mu\left(A_n\right)\leq r\iff\mu\left(\bigcap_{i=0}^\infty U_i\right)\geq 1-r\iff(\forall t)(\exists s)\;\mu\left(\bigcap_{i=0}^t U_i[s]\right)>1-r-2^{-t},\]
which is a $\PO{2}$--property. Hence the complementary property $\mu(A_n)>r$ is $\SO{2}$.
\end{proof}

In order to express the next result we need a generalization of the concept of being probabilistic. 
We call a function $f:\In(X,\delta_X)\mto(Y,\delta_Y)$ on represented spaces {\em limit probabilistic}
if there is a limit computable function $F:\In\IN^\IN\times2^\IN\to\IN^\IN$ and a family $(A_p)_{p\in D}$ of measurable sets $A_p\In2^\IN$
with $D:=\dom(f\delta_X)$ such that $\mu(A_p)>0$ for all $p\in D$, and $\delta_YF(p,r)\in f\delta_X(p)$ for all $p\in D$ and $r\in A_p$.
This concept is obviously weaker than being probabilistic, $\WKL$ is for instance not probabilistic by \cite[Proposition~14.8]{BGH15a},
but it is limit computable and hence limit probabilistic. 
However, we can prove the following.

\begin{lemma}
\label{lem:WKL-limit-probabilistic}
$\WKL'$ is not limit probabilistic.
\end{lemma}
\begin{proof}
The proof of \cite[Theorem~5.3]{JS72} relativizes and yields the following:
since there are $\SO{2}$--sets $A,B\In\IN$ that cannot be separated by a $\Delta_2^0$--set,
there is a tree $T\leqT\emptyset'$ such that
\[\mu\{D\in2^\IN:(\exists C\leqT D')\;C\in[T]\}=0.\]
The proof follows exactly along the lines of the proof given in \cite{JS72}, with the additional observation that the sets $C_i$
defined analogously to those in that proof are $\SO{2}$--sets by Lemma~\ref{lem:jump-measure}.
\end{proof}

Now we can derived the following corollary. 

\begin{corollary}
\label{cor:COH-probabilistic}
$\COH$ is not probabilistic.
\end{corollary}
\begin{proof}
Let us assume that $\COH$ is probabilistic. Then $\lim*\COH$ is limit probabilistic and
hence $\WKL'$ too by Corollary~\ref{cor:WKL-lim-COH}. This contradicts Lemma~\ref{lem:WKL-limit-probabilistic}.
\end{proof}

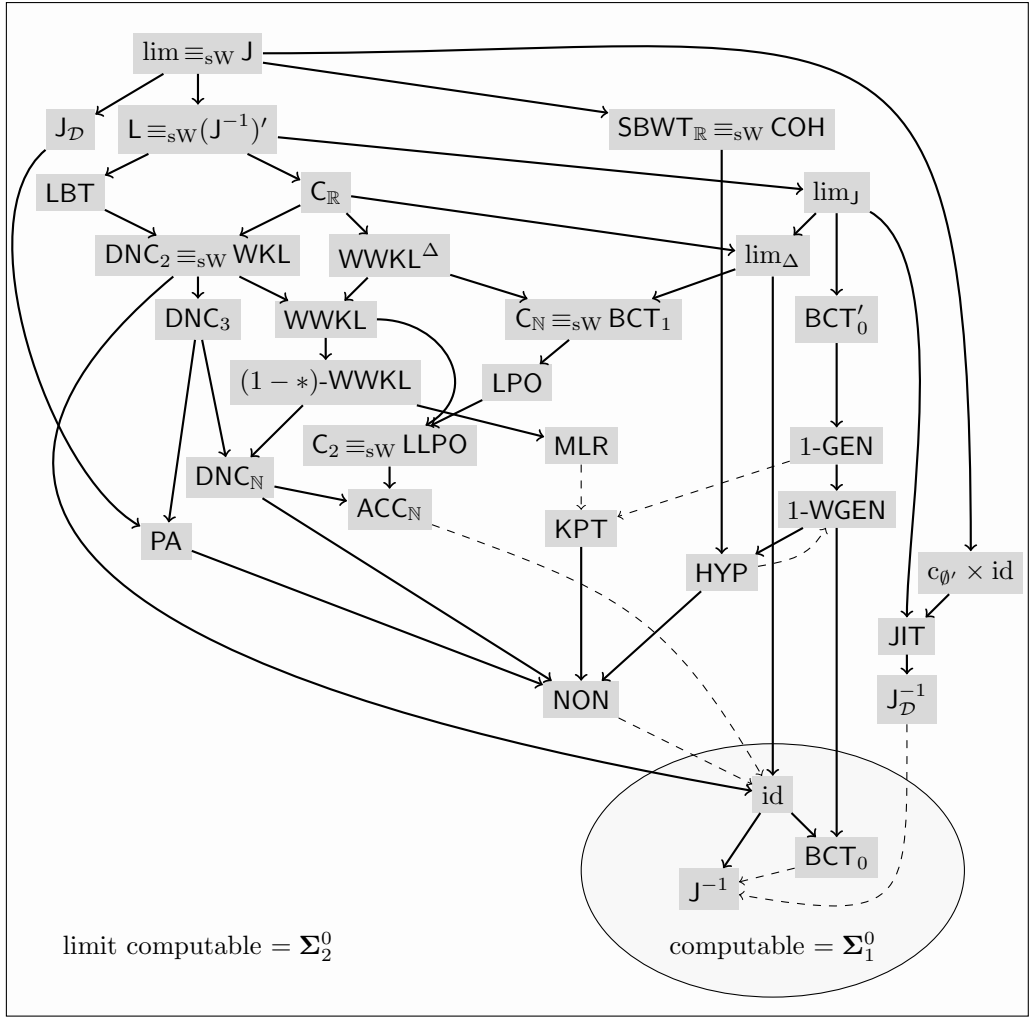
\begin{figure}[htb]
\begin{tikzpicture}[scale=0.84,auto=left,every node/.style={fill=black!15}]
\useasboundingbox  rectangle (16,16);
\def\rvdots{\raisebox{1mm}[\height][\depth]{$\huge\vdots$}};


\draw[style={fill=black!1}]  (0,16) rectangle (16,0);

\draw[style={fill=black!3}]  (12,2.3) ellipse (3 and 2);

  \node (COH) at (11.2,14) {$\SBWT_\IR\equivSW\COH$};
  \node (JD) at (1,14) {$\J_\DD$};
  \node (LIM) at (3,15.2) {$\lim\equivSW\J$};
  \node (LIMJ) at (13,13) {$\lim_\J$};
  \node (LIMD) at (12,12) {$\lim_\Delta$};
  \node (L) at (3,14) {$\Low\equivSW(\J^{-1})'$};
  \node (LBT) at (1,13) {$\LBT$};
  \node (CR) at (5,13) {$\C_\IR$};
  \node (WKL) at (3,12) {$\DNC_2\equivSW\WKL$};
  \node (DNC3) at (3,11) {$\DNC_3$};
  \node (DNCN) at (3.5,8.5) {$\DNC_\IN$};
  \node (PA) at (2.5,7.5) {$\PA$};
  \node (WWKLD) at (6,12) {$\WWKL^\Delta$};
  \node (WWKL) at (5,11) {$\WWKL$};
  \node (1SWWKL) at (5,10) {$(1-*)\dash\WWKL$};
  \node (MLR) at (9,9) {$\MLR$};
  \node (1GEN) at (13,9) {$1\dash\GEN$};
  \node (1WGEN) at (13,8) {$1\dash\WGEN$};
  \node (KPT) at (9,7.7) {$\KPT$};
  \node (BCT0) at (13,2.5) {$\BCT_0$};
  \node (BCT0S) at (13,11) {$\BCT_0'$};
  \node (HYP) at (11.2,7) {$\HYP$};
  \node (CN) at (9.2,11) {$\C_\IN\equivSW\BCT_1$};
  \node (LPO) at (8,10) {$\LPO$};
  \node (LLPO) at (6,9) {$\C_2\equivSW\LLPO$};
  \node (ACCN) at (6,8) {$\ACC_\IN$};
  \node (NON) at (9,5) {$\NON$};
  \node (ID) at (12,3.5) {$\id$};
  \node (JM) at (11,2) {$\J^{-1}$};
  \node (JD1) at (14.1,5) {$\J_\DD^{-1}$};
  \node (JIT) at (14.1,6) {$\JIT$};
  \node (cKid) at (15.1,7) {$\cc_{\emptyset'}\times\id$};

 
 \foreach \from/\to in 
 {
 LIM/COH,
  COH/HYP,
  LIM/L,
  LIM/JD,
  L/LBT,
  L/CR,
  LBT/WKL,
  CR/WKL,
  CR/WWKLD,
  WKL/WWKL,
  WKL/DNC3,
  DNC3/DNCN,
  DNCN/ACCN,
  WWKLD/WWKL,
  WWKL/1SWWKL,
  1SWWKL/DNCN,
  LLPO/ACCN,
  DNCN/NON,
  DNC3/PA,
  PA/NON,
  WWKLD/CN,
  CN/LPO,
  LPO/LLPO,
  1SWWKL/MLR,
  L/LIMJ,
  LIMJ/BCT0S,
  LIMJ/LIMD,
  LIMD/CN,
  BCT0S/1GEN,
  1WGEN/BCT0,
  1GEN/1WGEN,
  1WGEN/HYP,
  KPT/NON,
  HYP/NON,
  ID/BCT0,
  LIMD/ID,
  CR/LIMD,
  ID/JM,
  cKid/JIT,
  JIT/JD1}
 \draw [->,thick] (\from) -- (\to);

 
 \foreach \from/\to in 
 {NON/ID,
  1GEN/KPT,
  MLR/KPT, 
  BCT0/JM}
 \draw [->,dashed] (\from) -- (\to);


\draw [->,thick,looseness=1.5] (WWKL) to [out=0,in=30] (LLPO);
\draw [->,thick,looseness=1.5] (WKL) to [out=220,in=170] (ID);
\draw [->,thick,looseness=0.7] (JD) to [out=220,in=150] (PA);
\draw [->,thick,looseness=0.6] (LIMJ) to [out=330,in=90] (JIT);
\draw [->,thick,looseness=1.8] (LIM) to [out=0,in=90] (cKid);


\draw [->,dashed,looseness=1.1] (HYP) to [out=10,in=240] (1WGEN);
\draw [->,dashed,looseness=1.2] (ACCN) to [out=340,in=120] (ID);
\draw [->,dashed,looseness=1.9] (JD1) to [out=270,in=350] (JM);


\node[style={fill=black!1}] at (3,1.1) {limit computable $=\SO{2}$};
\node[style={fill=black!3}] at (12,1.1) {computable $=\SO{1}$};

\end{tikzpicture}
\ \\[-0.5cm]
\caption{Part of the computability theory zoo in the Weihrauch lattice. The solid arrows indicate strong Weihrauch reductions in the opposite direction and the dashed
arrows indicate ordinary Weihrauch reductions.}
\label{fig:diagram-Full}
\end{figure}

\section{Conclusion}

In this paper we have started to classify the uniform computational content of computability theory with the
help of the Weihrauch lattice.
The diagram in Figure~\ref{fig:diagram-Full} visualizes some of the reductions that we have studied.
While the non-relativized versions of problems that do not depend on parameters can be studied in the Medvedev lattice,
and in the less uniform Muchnik lattice too \cite{Sim15}, there are some significant differences. 
Perhaps most noticeable is that the problems $\DNC_\IN$ and $\MLR$ are incomparable in our lattice,
whereas the problem corresponding to $\DNC_\IN$ in the Muchnik lattice is reducible to the problem corresponding to $\MLR$ \cite{Sim15}.

While we have studied a number of relevant computability-theoretic properties, we have only started
to look at some of the most basic (non-constructive) theorems of computability theory. 
It is a promising new research programme to continue along these lines and to study
more advanced theorems that require priority constructions and other techniques whose
uniform computational content has not yet been investigated.

\section{Acknowledgments}

We would like to thank Arno Pauly for helpful comments on an earlier version of this article
and the anonymous referee for her or his very careful proof reading that helped us to improve
some results and the presentation of the article.

\bibliographystyle{plain}

\bibliography{C:/Users/vbrattka/Dropbox/Bibliography/lit}

\end{document}